\documentclass[12pt,a4paper]{article}

\usepackage[english]{babel}
\usepackage[utf8x]{inputenc}
\usepackage[T1]{fontenc}
\usepackage[a4paper,top=2cm,bottom=2cm,left=1.5cm,right=1.5cm,marginparwidth=1.75cm]{geometry}

\usepackage{amsmath,amsthm,amsfonts}
\usepackage{amssymb}
\usepackage{mathtools}
\usepackage{fancyhdr}
\usepackage{graphicx}
\usepackage{stmaryrd}
\usepackage[colorinlistoftodos]{todonotes}
\usepackage[colorlinks=true, allcolors=blue]{hyperref}
\usepackage{scalerel}
\usepackage{mathrsfs}
\usepackage{relsize}
\usepackage{enumerate}   
\usepackage{bm} 
\usepackage{dsfont}

\newtheorem{theorem}{Theorem}[section]
\newtheorem{lemma}[theorem]{Lemma}
\newtheorem{proposition}[theorem]{Proposition}
\newtheorem{corollary}[theorem]{Corollary}
\theoremstyle{definition}
\newtheorem{definition}[theorem]{Definition} 

\newtheorem{remark}[theorem]{Remark}

\theoremstyle{remark}

\makeindex             
\usepackage{mathtools}
\usepackage{tikz}
\usetikzlibrary{positioning}
\usetikzlibrary{arrows}

\newcommand\wst{ w^{*} }   
\def\LpRdN{\big( \LpRd, \, \|\ebbes\|_p \big)}   
\def\Rdst{{{\Rst^d}}}   
\def\BspN{(\Bsp, \, \|\ebbes\|_\Bsp)}   
\def\Bsp{{\boldsymbol B}}   
\newcommand\sumkinf{\sum_{k=1}^\infty}   
\newcommand\sumkin{\sum_{k=1}^n}   
\def\Cst{{\mathbb C}}   
\def\bb{ {\bf b}}   
\newcommand\BPsp{{ \Bsp'}}   
\newcommand\BPN{{ (\BPsp, \| \ebbes \|_\BPsp)}}   
\def\opnormi#1{{|\nnth \| {#1} | \nnth \|  \, }}   
\newcommand\Btsp{{\Bsp^2}}   
\newcommand\Bisp{{\Bsp^1}}   
\def\ScRd{{\Scsp(\Rst^d)}}   
\def\CcRd{{\Ccsp(\Rst^d)}}   
\def\supp{\operatorname{supp}}   
\def\intRd{\int_{\Rst^d}}   
\def\LpRdN{\big( \LpRd, \, \|\ebbes\|_p \big)}   
\def\LtRd{{\Ltsp(\Rst^d)}}   
\newcommand\Rd{\Rdst}   
\newcommand{\Strho}{{\operatorname{St}_\rho}}   
\newcommand{\bZ}{\mathbb{Z}}        
\def\CbRd{{\Cbsp(\Rdst)}}   
\def\CbRdN{{\big( \Cbsp(\Rst^d), \, \|\ebbes\|_\infty \big)}}   
\def\CORd{{\COsp(\Rst^d)}}   
\def\MbRd{{\Mbsp(\Rst^d)}}   
\def\LpRd{{\Lpsp(\Rst^d)}}   
\def\ebbes{\mbox{$\,\cdot\,$}}   
\def\Rst{{\mathbb R}}   
\newcommand\LBit{{ {\mathcal L}(\Bisp \negthinspace , \negthinspace \Btsp)}}   
\def\nnth{{ \negthinspace \: \negthinspace }}   
\def\ScRd{{\Scsp(\Rst^d)}}   
\def\LtR{{\Ltsp(\Rst)}}   
\def\CcRd{{\Ccsp(\Rst^d)}}   
\def\MbRd{{\Mbsp(\Rst^d)}}   
\newcommand{\RR}{\mathbb{R}}   
\def\CcRd{{\Ccsp(\Rst^d)}}   
\def\Mbsp{{\Msp_{\negthinspace b}}}   
\def\COsp{{\Csp_{\negthinspace 0}}}   
\def\CORd{{\COsp(\Rst^d)}}   
\def\CbRd{{\Cbsp(\Rdst)}}   
\def\Cbsp{{\Csp_{\negthinspace b}}}   
\def\Csp{{\boldsymbol C}}   
\def\osc{\operatorname{osc}}
\def\oscd{\osc_{\delta}}
\newcommand{\epso}{{ \varepsilon > 0 }}   
\def\Msp{{\boldsymbol M}}   
\def\Ccsp{{\Csp_{\negthinspace c}}}   
\newcommand\MbRdN{{(\Mbsp(\Rst^d), \| \ebbes \|_\Mbsp )}}   
\def\Wsp{{\boldsymbol W}}   
\def\Zdst{{\Zst^d}}   
\def\CORdN{{\big( \COsp(\Rst^d), \, \|\ebbes\|_\infty \big)}}   
\def\FT{{\operatorname{{\mathcal F}}}} 
\newcommand{\RRn}{{\RR^n}}   
\newcommand{\RRm}{{\RR^m}}   
\def\SORd{{\SOsp(\Rst^d)}}   
\def\SOR{{\SOsp(\Rst)}}   
\def\SORN{\big( \SOR, \|\ebbes\|_\SOsp \big)}   
\def\SORdN{\big( \SORd, \|\ebbes\|_\SOsp \big)}   
\def\SOsp{{\Ssp_{\negthinspace 0}}}   
\def\Ssp{{\boldsymbol S}}   
\newcommand{\SO}{\SOsp}   
\newcommand\Rtd{ {{\Rst}^{2d} }}   
\def\SOPRd{{\SOPsp(\Rst^d)}}   
\def\SOPsp{{\Ssp_{\negthinspace 0}'}}   
\newcommand\LLt{ { \mathcal L} (\Ltsp,\Ltsp)}   
\def\SORtd{{\SOsp(\Rst^{2d})}}   
\newcommand{\TFd}{{{  \Rdst \times \Rdsth }}}   
\def\Rdsth{{\widehat{\Rst}^d}}   
\newcommand\HS{ {\cal HS}}   
\def\SOPRtd{{\SOPsp(\Rst^{2d})}}   
\def\Lpsp{{\Lsp^p}}   
\newcommand\BtspN{(\Bsp^2, \, \|\ebbes\|^{(2)})}  
\def\LtR{{\Ltsp(\Rst)}}   
\def\LtRd{{\Ltsp(\Rst^d)}}   
\def\LiRd{{\Lisp(\Rst^d)}}   
\newcommand{\Drho}{{\operatorname{D}_{\rho}}}   
\newcommand{\bC}{\mathbb{C}}        
\newcommand{\ofp}[1]{{\slp{#1}\srp}} 
\def\slp{{{\raise 0.5pt \hbox{\footnotesize $($}}}}   
\def\srp{{{\raise 0.5pt \hbox{\footnotesize $)$}}}}   
\def\SpPsi{\Sp_\Psi}   
\def\Sp{\operatorname{Sp}}   
\def\Zst{{\mathbb Z}}   
\newcommand\sumnZd{\sum_{n\in\Zdst}}   
\def\Lsp{{\boldsymbol L}}   
\def\Lisp{{\Lsp^1}}   
\newcommand{\bR}{\mathbb{R}}        
\newcommand{\bN}{\mathbb{N}}        
\def\hatf{\hat f}
\def\fhat{\hat f}
\def\IFT{\operatorname{\mathcal F}^{-1}}   
\def\Trans{\operatorname{T}}
\newcommand{\WRN}{{\big( \WR, \, \|\ebbes\|_\Wsp \big)}}   
\def\Scsp{{\boldsymbol{\mathcal S}}}   
\def\Ltsp{{\Lsp^2}}   
\newcommand{\infnorm}[1]{{\lVert #1 \rVert_\infty}}   
\newcommand\LinfRd{ L^{\infty}(\Rst^{d})}   
\newcommand{\WR}{{\Wsp(\Rst)}}   
\def\Rtdst{{\Rst^{2d}}}   
\newcommand\Rm{{  {\Rst}^m }}   
\def\Asp{{\boldsymbol A}}   
\def\NN{\mathbb{N}}
\newcommand{\RRd}{{\mathbb{R}^d}}   
\def\Shah{{\makebox[2.3ex][s]{$\sqcup$\hspace{-0.15em}\hfill $\sqcup$}\, \, }}   
\def\xx{ {\bf x}}   
\newcommand\kin{{_{k=1}^n}}   
\def\ee{ {\bf e}}   
\def\tensor{\otimes}   
\def\LtRtd{{\Ltsp(\Rst^{2d})}}   
\def\trace{{\operatorname{trace}}}   
\def\COPRd{{\COPsp(\Rst^d)}}   
\def\SOPRdN{(\SOPRd , \| \ebbes \|_{\SOPsp} ) }   
\def\WRd{{\Wsp(\Rdst)}}
\def\WRdN{{(\WRd, \| \ebbes \|_\Wsp) }}

\def\LtRdN{\big( \LtRd, \, \|\ebbes\|_2 \big)}   
\newcommand\LSOSOP{ {\cal L} (\SOsp,\SOPsp)}   
\newcommand{\inv}{^{-1}}   
\def\LtRdN{\big( \LtRd, \, \|\ebbes\|_2 \big)}   
\def\sinc{\operatorname{sinc}}   
\def\LtRN{\big( \LtR, \, \|\ebbes\|_2 \big)}   
\def\SINC{{\operatorname{sinc}}} 
\def\LiR{{\Lisp(\Rst)}}   
\def\CORN{{\big( \COsp(\Rst), \, \|\ebbes\|_\infty \big)}}   
\newcommand{\FLiRd}{{ \FLi(\Rdst) }}   
\def\ScPRd{{\ScPsp(\Rst^d)}}   
\def\ScPsp{{\Scsp'}}   
\def\COPsp{{\Csp'_{\negthinspace 0}}}   
\newcommand\LSOPwSO{ {\cal L}_{\wst} (\SOPsp,\SOsp)}   
\def\LqRdN{\big( \LqRd, \, \|\ebbes\|_q \big)}   
\def\Nst{{\mathbb N}}   
\newcommand\HLiSOSOP{{H_\Lisp(\SOsp,\SOPsp)}}   
\def\checkm{{^\checkmark}}   
\newcommand{\WCOliRd}{\WTsp \COsp  \lisp (\Rdst) }   
\def\LpG{{\Lpsp(G)}}   
\newcommand\FLpsp{{\FF \negthinspace \Lpsp}}   
\newcommand\FLi{{\mathcal F}{\negthinspace \Lisp}}   
\newcommand{\intinf}{\int_{-\infty}^{\, \infty}}   
\def\ltsp{{\lsp^2}}   
\def\lsp{{\boldsymbol\ell}}   
\def\LiRdN{\big( \LiRd, \, \|\ebbes\|_1 \big)}   
\def\WTsp#1#2{{\Wsp(#1,#2)}}   
\newcommand\BspqRdN{{ (\BspqRd, \| \ebbes \|_{\Bspq}) }}   
\newcommand{\FF}{{\mathcal{F}}}   
\newcommand{\LqRd}{{ \Lsp^q(\Rdst) }}
\newcommand{\SOGTrRd}{{ (\SOsp,\Ltsp,\SOPsp)(\Rdst) }}   

\def\BspqRd{{\Bspqsp(\Rst^d)}}   
\newcommand\lqsp{{ \lsp^q}}   
\def\lisp{{\lsp^1}}   
\newcommand\FLiRdN{\big( \FLiRd, \, \|\ebbes\|_{\FLisp} \big)}   
\def\Bspqsp{{\Bsp^{s}_{\negthinspace p,q}}}   
\newcommand{\Bspq}{{\Bsp^s_{\negthinspace p,q}}}   
\newcommand\WFLiliRd{{\Wsp(\FLi,\lisp)(\Rdst)}}   
\def\WFW{\Wsp_{\negthinspace \negthinspace \FT}}
\def\WFWRd{{\WFW \!\,  (\Rdst)}}  

\newcommand\WFWRn{\WFW(\bR^{n})}
\newcommand\weaks{weak$^{*}$}
\newcommand{\FLisp}{{ {\mathcal F}\Lisp}}   

\def\Rn{{\Rdst}}  

\def\wedge{\Delta}  

\title{Distribution Theory by Riemann Integrals}
\author{Hans G.\ Feichtinger\thanks{Faculty of Mathematics,
 Univ.\  Vienna,  Oskar-Morgenstern-Platz 1, 1090 Wien, AUSTRIA, and Charles Univ. Prague, \mbox{E-mail: \protect\url{hans.feichtinger@univie.ac.at}}} and Mads S.\ Jakobsen\thanks{Norwegian University of Science and Technology, Department of Mathematical Sciences, Trondheim, Norway, \mbox{E-mail: \protect\url{mads.jakobsen@ntnu.no}}}}

\begin{document}
\date{\today}
\maketitle

\begin{abstract}It is the purpose of this article to outline a syllabus for a course that can be given to engineers looking for an understandable mathematical description of the foundations of distribution theory and the necessary functional analytic methods. Arguably, these are needed for a deeper understanding of basic questions in signal analysis. Objects such as the Dirac delta and the Dirac comb should have a proper definition, and it should be possible to explain how one can reconstruct a band-limited function from its samples by means of simple series expansions. 
It should also be useful for graduate mathematics students who want to see how functional analysis can help to understand fairly practical problems, or teachers who want to offer a course related to the ``Mathematical Foundations of Signal Processing'' at their institutions. \newline 
The course requires only an understanding of the basic terms from linear functional analysis, namely  Banach spaces and their duals, bounded linear operators and a simple version of $\wst$-convergence. 
As a matter of fact we use a set of function spaces which is quite different from the collection of Lebesgue spaces $\LpRdN$ used normally. We thus avoid the use of Lebesgue integration theory. Furthermore we avoid topological vector spaces in the form of the Schwartz space. \\ \indent Although practically all the tools developed and presented can be realized in the context of LCA (locally compact Abelian) groups, i.e.\ in the most  general setting where a (commutative) Fourier transform makes sense, we restrict our attention in the current presentation to the Euclidean setting, where we have (generalized) functions over $\Rdst$. This allows us to make use
of simple BUPUs (bounded, uniform partitions of unity), to apply dilation operators and occasionally to make use of concrete special functions such as the (Fourier invariant) standard Gaussian, given by $g_0(t) = \exp(- \pi \vert t \vert^{2})$. \\ \indent
The problems of the overall current situation, with the separation of theoretical Fourier Analysis as carried out by (pure) mathematicians and Applied Fourier Analysis (as used in engineering applications) are getting bigger and bigger and therefore courses filling the gap are in strong need.
This note provides an outline and may serve as a guideline. The first author has given similar courses over the last years at different schools (ETH Z\"urich, DTU Lyngby, TU Muenich, and currently Charlyes University Prague) and so one can claim that the outline is not just another theoretical contribution to the field.
\end{abstract}


\section{Overall Motivation}

\subsection{Psychological Aspects}

It is not a secret that the way how engineers or physicists are describing ``realities'' 
is quite different from the way  mathematicians want to
describe the same thing. The usual agreement is that applied scientists are motivated by the concrete applications and therefore do not need to be so pedantic in the description,
because they have a ``better feeling''  about what is true and what is not true. After all, it does not pay to be too pedantic if one wants to make progress.

On the other hand mathematicians have a tendency to be too formal, to consider
formal correctness of a statement as more important than the possible usefulness of a statement, simply because usefulness is not a category in mathematical sciences. 
Applicability by itself is not a criterion for important
mathematical results which often go for the details of a structure without taking care of its relevance for applications. Sometimes this ``abstract viewpoint''
is very helpful, because it reveals important, underlying structures or allows to find connections between fields which appear to have very little in common at first sight. However, in the right (abstract) mathematical model they appear
to be almost identical. Such observations allow to sometimes transfer information and insight, or computational rules established in one area to another area, which certainly is not possible if only one single application is in the focus.

There are different ways to view these discrepancies. What we could call the
{\it negative attitude} is to say as a mathematician: {\it You know, engineers
and physicists are extremely sloppy, you never can trust
their formulas. They claim to derive mathematical identities
by using divergent integrals and so on, so one has to be careful in taking over what they ``prove''}.  In the same way the engineer might say: 
{\it You know, mathematicians are pedantic people
who care only about technical details and not for the content of a formula.
Whenever they claim that our formulas are not correct they find after some while a way to produce more theory in order to then prove that our formulas have been correct after all.} 

A more positive and ambitious approach would be to agree from both sides on a few facts which are on average quite valid:
\begin{itemize}
\item Any mathematical statement should, at least at the end, have a proper
   mathematical justification;
\item Formulas developed from applied scientists may, at least at the beginning,
   come from intuition or experiments, so they might be valid under particular
   conditions or under implicit assumptions (which are often clear from the
   physical context, e.g.\ positivity assumptions, etc.);
\item For the progress new formulas might be more important than a refined
   analysis of established formulas, but the goal is to have
   {\it useful formulas whose range of applications} (the relevant assumptions)
   are well understood; it is important to know when there is a guarantee that
   the formula can be applied (because there is a proof), and when one might
   be at risk of getting a wrong result 
   (even if it is with low probability);
\item  This goal requires cooperation between applied scientists and mathematicians;
   usually the first group is better trained in establishing unexplored problems while the second is expected to 
   provide a theoretical setup which ensures that
   things are under control, in terms of correctness of assumptions and    conclusions. Obviously, in an ideal
   world one group can and should learn a lot from the other.
\end{itemize}

So in the cooperation between the two communities mathematicians should
learn more about {\it the goals and the motivation} and e.g.\ {\it engineers and physicists} might learn that it is also beneficial to cooperate with
mathematicians and to have clear guidelines concerning the correct use of formulas and mathematical identities and where perhaps caution is in place.   

\subsection{The search for a Banach space of test functions}

The overall goal of this paper is to propose a path
that allows us to introduce a family of {\it generalized functions}
which is large enough to contain most of those generalized 
functions which are relevant in the context of (abstract or
applied) Fourier analysis and for engineering applications.
Specifically Dirac measures and Dirac combs. 
We will demonstrate that this is possible using modest tools
from functional analysis.

Before going to the technical side of the exposition
let us motivate the use of dual spaces and functional analytic methods, and shed some light on the idea of {\it distributions}.
Let us start with some observations:
\begin{itemize}
    \item First of all it is clear that {\it generalized functions} should form a linear space, so that linear combinations of those objects (sometimes called signals) can be formed, and under
    certain conditions, even limits, and hence infinite series;
    \item Secondly we would like to have ``ordinary functions''
    included in a natural way within the world of generalized
    functions, so we need a {\it natural embedding} of
    as many linear spaces of ordinary functions as possible;
    \item As a third variant 
    we can think of generalized functions as a kind
    of ``limits'' of ordinary functions, but in a specific
    sense (and ideally the convergence should also allowed 
    to be applied to the generalized functions); 
    \item Finally there are many operations that can be carried
    out for (certain) functions, such as translation,
    convolution, dilation, Fourier transform, and we will
    go for a setting where the approximation properties
    of the previous item allow to extend these operations
    to the linear space of generalized functions. 
\end{itemize}

In order to explain our understanding of ``distribution theory'' let
us first formulate again some general thoughts. In fact 
it is not surprising, 
that we have to use functional analytic methods in this context 
because after all at least for continuous variables signal spaces tend to be {\it not finite-dimensional} anymore\footnote{Commonly
the term ``infinite dimensional'' is used, and we will also use
it later on, but this expression wrongly suggests that instead
of a {\it finite basis} one just has an infinite basis, and this
is not what we should expect or use!   }  and so we have to 
resort to methods that allow us to describe the convergence
of infinite series. The simplest way to do this is to 
assume that one has a linear space  and a normed space, $\BspN$.
If one has in addition a kind of multiplication
$(a,b) \mapsto a \bullet b$ (with the usual rules) 
one speaks of {\it normed algebras}, if
$$ \|b_1 \bullet b_2\|_\Bsp  \leq \|b_1\|_\Bsp \cdot \|b_2\|_\Bsp \ \ \text{for all} \ \ b_1, b_2 \in \Bsp.$$ 

Among the normed spaces those which are {\it complete}, the
{\it Banach spaces} are the most important ones, because like
$\Rst$ itself with the mapping $x \mapsto |x|$ one has
(by definition) {\it completeness}, meaning that every
{\it Cauchy sequence} is convergent. This is known to be
equivalent to the fact that every 
{\it absolutely convergent sequence} with 
$\sumkinf \|b_k\|_\Bsp  < \infty$, is convergent, 
so that the partial sums
$\sumkin b_k$ have a limit (in $\BspN$). Therefore the infinite sum is (unconditionally, or independent
of the order) well defined, and thus the symbol 
$\sumkinf b_k$ is meaningful in this situation. 

The most important tool within linear functional analysis
are the {\it linear functionals}, or bounded
linear mappings from $\Bsp$ into $\Cst$ (or into $\Rst$
for the case of real vector spaces). Such a functional $\sigma$
has to satisfy two properties:
\begin{enumerate}
    \item Linearity:  $\quad 
    \sigma( \alpha \bb_1 + \beta \bb_2)
      =  \alpha \sigma(\bb_1) + \beta \sigma(\bb_2), \quad  \bb_1,\bb_2 \in \Bsp, \alpha, \beta \in \Cst.$
    \item Boundedness: There exists $c > 0 $
    such that $\quad |\sigma(\bb)| \leq C \|\bb\|_\Bsp, \,\, 
    \forall \bb \in \Bsp.$
\end{enumerate}

For any given normed space $\BspN$
the collection of all such bounded linear functionals constitutes
the {\it dual space}, denoted by $\BPsp$. It carries a  norm, given by
$$
 \|\sigma\|_\BPsp :=  \sup_{\|\bb\|_\Bsp \leq 1} |\sigma(\bb)|.
$$
With this norm $\BPsp$ turns out to be a
Banach space\footnote{Even if 
$\BspN$ is just a normed space.}.  One can think of the
dual space as the collection of all coordinate functionals
(describing the contribution of a fixed element in a basis)
over all finite dimensional subspaces of $\Bsp$, thus capturing
all the information about the underlying normed space. 

In addition to norm convergence on $\BPsp$ we will use what is called
the $\wst$-convergence. It can
be described for sequences as {\it convergence in action}:

For all practical purposes\footnote{Technically speaking, for {\it separable}  Banach
spaces $\BspN$ which are , which contain a 
countable, dense subset. Thus will be the case for all the situatios where we make use of this concept.} the following definition
is a simple way of describing what is called  $\wst$-convergence.
\begin{definition}
A sequence of linear functionals $(\sigma_n)_{n \geq 1}$
{\it converges in action} or {\it in the weak$^{*}$-sense} to some $\sigma_0 \in \BPsp$ if we have 
\begin{equation}
    \lim_{n \to \infty} \sigma_n(\bb) = \sigma_0(\bb) \ \ \text{for all} \ \ \bb \in \Bsp. 
\end{equation}
\end{definition}


By the Banach-Steinhaus Theorem the convergence 
for {\it all} $\bb \in \Bsp$ implies boundedness,
i.e.\ 
$\sup_{n \geq 1} \| \sigma\|_\BPsp < \infty,$
and that conversely it is (under this condition!) enough to claim that the limits on the  left hand side exist for any $\bb \in \Bsp$, thus defining the functional $\sigma_0$. 
In fact, it would be even enough (given the boundedness condition) to know that one has a limit for all $b$ from a dense subspace of $\BspN$. 

Infinite dimensional Banach spaces $\BspN$ do not satisfy the Heine-Borel
property. A bounded sequence may fail to have a (norm) convergent subsequence. But the {\it Banach-Alaoglou Theorem} 
(see \cite{co90})  
ensures that any bounded sequence $(\sigma_k)$ in $\BPN$
has a subsequence $(\sigma_{n_k})_{k \geq 1}$ which is
$\wst$-convergent to some $\sigma_0 \in \BPsp,$ i.e.\ 
$$ \lim_{k \to \infty} \sigma_{n_k}(b) = \sigma_0(b) \ \ \text{for all} \ \ \bb \in \Bsp.$$ 

In a similar way the set of all bounded and linear operators
between two normed spaces is defined, we denote it by $\LBit$.
It is always a normed space with respect to the operator norm
$$\opnormi{T} := \sup_{\|b_1\|_\Bisp \leq 1} \|T(b_1)\|_\Btsp. $$
and if $\BtspN$ is a Banach space the space of 
operators is complete as well. 
In particular, for the choice $\Btsp = \Cst$ the space reduces to the dual space. 

For the case $\Bisp = \Bsp = \Btsp$ these operators form a normed algebra, and in fact a 
{\it Banach algebra} if $\BspN$ is a Banach space. 

Since many sequences of functions which do not have a reasonable
pointwise limit, such as a sequence of compressed box-functions 
which converge to the so-called {\it Dirac Delta}, often
denoted by $\delta(t)$ in the engineering literature, are in
fact limits in this sense, it is at least plausible to 
work with dual spaces in order to capture these limits. 

Without going too much into the psychological and didactical
side of this issue let us just state here that indeed, it is 
{\it meaningful} to model generalized functions as what we
will call {\it distributions}, namely elements of dual
spaces for suitable chosen Banach spaces $\BspN$ of
integrable and bounded, continuous functions. 

We admit that of course this terminology is influenced
by the existing traditional way of introducing generalized
functions, e.g.\ by using the {\it tempered
distributions} developed by Laurent Schwartz (\cite{sc57})
using the (nuclear Frechet) space $\ScRd$ of {\it rapidly
descreasing functions}. While  differentiability is 
in the focus of attention there, we leave this aspect aside and
allow ourselves to call an algebra (with respect to 
pointwise multiplication and/or convolution) of continuous
functions a {\it space of test functions} and the dual space
a space of {\it distributions}. This will be the setting
we choose for our approach. Thus from now on we will mostly
talk about test functions and distributions,
but we will still have to explain in which sense distributions
are generalized functions in the spirit of the above description.

One can also motivate the use of dual spaces for the description
of linear spaces of signals by the following argument: \newline
\centerline{ 
{\it A signal is something that can be measured!}}
Just thinking of an audio signal which we can record 
using a microphone, we can compress using MP3 coding 
based on the FFT, and we can transmit it. All this is
on the basis of {\it linear measurements} which are of
course continuous in some sense, meaning that quite similar
signals (whatever they are) will provide similar measurements.
But is the audio signal a pointwise almost everywhere defined 
function in $\LtR$ in the mathematical sense?  Of course 
we can take pictures of a natural scene and enjoy the 
quality of color picture taken by a $16$-million pixel camera, 
but does that device really sample (in the mathematical sense)
a continuous, $2D$-function describing the analog picture
which we use in a conversational situation? 

The situation is really much more like an abstract 
probability distribution, say a normal distribution with
some expectation value and some variance. We will never
be able (except through indirect mathematical description)
to provide a pointwise description of such a ``distribution''
(a different but related use of this word), so normally
one resorts to the use of {\it histograms}. Given the 
bins used for the histogram one can describe the height 
of the bars simply as the value obtained by applying the (non-negative) measure (via integration) to the indicator
function of the corresponding interval (bin), making 
sure that the union of the bins is the whole real line
or at least the range of the random variable resp.\
the support of the corresponding measure. 

What we are doing here is essentially to replace
those (finer and finer) bins by BUPUs (uniform partitions
of unity), with the extra demand of assuming that
they are continuous and not just step functions.
The reader should see this as a minor and just technical
modification (which is avoiding the distinction between
step functions and continuous functions, and is also
much more convenient for the setting of LCA groups).

The (abstract) viewpoint of considering signals as something that can be measured also suggests very naturally a measure of similarity of signals. If for a given (potentially comprehensive) set of measurements only very small deviations are observed, then we think of those signals as ``quite similar'', and a sequence of signals may converge in this way to a limit signal (e.g.\ coarse approximations to the continuous limit). But this kind of convergence is encapsulated mathematically in the concept of $\wst$-convergence described above, that will be used intensively in this text. 

\section{Notations and Preliminaries}  

Although the approach described below can be used to develop
Harmonic Analysis in the context of locally compact Abelian (LCA) groups we restrict our attention to the setting of Euclidean 
spaces $\Rdst$. This is the framework relevant for most engineering work and physics. 

Let us fix some notation. It all starts with the most simple vector space of functions
on $\Rdst$, namely $\CcRd$, the space of continuous, complex-valued  and compactly supported functions on $\Rd$, i.e.\ 
with $\supp(k) \subset B_R(0) := \{ x \, : \,  |x| \leq R \}$ for some $R>0$. 
For such a function $f\in\CcRd$ the notion of an integral, $\intRd f(t) \, dt$, is well-defined by Riemann integration, and thus 
this (infinite-dimensional) linear space of functions can be
endowed with many different norms, such as the maximum-norm or uniform-norm,
$\|k\|_\infty = \sup_{t \in \Rdst} |f(t)| $ and 
the $p$-norms $\|k\|_p = (\intRd|k(t)|^p\,dt)^{1/p}$ for
$ 1 \leq p < \infty$.  The \emph{completion} of $\CcRd$ with respect to the $p$-norm yields the Lebesgue spaces, $\LpRdN$. Most notably are $\LiRd$ and $\LtRd$. The latter being a Hilbert space with respect to the inner-product $\langle f,g\rangle = \intRd f(t) \, \overline{g(t)} \, dt$.

\noindent For complex-valued functions $f,g$ on $\Rd$ we define the following operations,
\begin{enumerate}
    \item[] point-wise multiplication, $(f\cdot g)(t) = f(t) \cdot g(t)$, $t\in\Rd$,
    \item[] flip operation, $f^{\checkmark}\!(t) = f(-t)$,
    \item[] complex conjugation, $\overline{f}(t) = \overline{f(t)}$,
    \item[] translation by $x\in\Rd$, $T_xf(t) = f(t-x)$,
    \item[] modulation by $\omega\in\Rd$, $E_{\omega}f(t) = e^{2\pi i \omega\cdot t} \, f(t)$,  
    \item[] dilation by an invertible $d\times d$ matrix $A$, $\alpha_{A}f(t) = \vert \det (A) \vert^{1/2}\, f(At)$,
    \item [] specifically homogeneous dilations for $\rho > 0$,
     \newline
     $\quad [\Strho f](t) =  \rho^{-d} f(t/\rho), $  and 
     $[\Drho h](t) = h(\rho t)$ \newline
     with $\|  \Strho f\|_1 = \|f\|_1 $
     and $\| \Drho f\|_\infty = \|f\|_\infty.$
 \end{enumerate}

 \noindent Let $\wedge$ be the \emph{tent}-function given  by
\[ \wedge(t) = \prod_{j=1}^{d} \max\big( 1 - 2 \vert t^{(j)} \vert , 0 \big), \ \ t = (t^{(1)},t^{(2)},\ldots, t^{(d)})\in\Rd.\]
Observe that $\supp\, \wedge = [-1/2,1/2]^{d}$. We define the family of functions $(\psi_{n})_{n\in\bZ^{d}}$ 
to be the collection of half-integer translates of $\wedge$, so that 
 \begin{equation} \label{eq:our-bupu} \psi_{n}(t) = \wedge\big(t-\tfrac{1}{2}n\big), \ t\in\Rd, \ n\in\bZ^{d}.\end{equation}
The crucial properties of the functions $(\psi_{n})$ are for us that they satisfy the general assumptions of a \emph{bounded uniform partition of unity} (BUPU), of which we give the definition below. Throughout this work $(\psi_{n})$ will always refer to the functions in \eqref{eq:our-bupu}. However, any other BUPU can also be used, which entails only minor modifications to our proofs. 

For most applications {\it regular BUPUs} will be sufficient (and easier to handle), which are obtained as translates of one (smooth) function with compact support along some lattice in 
$\Rdst$. In this setting it is natural to use smooth BUPUs with respect to some lattice $\Lambda = {\bf A} \Zdst$, for some non-singular $d \times d$  matrix ${\bf A}$. For convenience of notation we use mostly lattices of the form $\gamma \Zdst$,
for some $\gamma > 0$. 
\begin{definition} \label{regFL1BUPU}
A family $\Psi = (\psi_k)_{k \in \Zdst} = (T_{\gamma k} \psi_0)_{k \in \Zdst} $ in $\CcRd$ (for some $\gamma > 0$) 
 is called a {\it regular, uniform
 partition of unity}
 on $\Rdst$ of size $R$,
(we write $\vert \Psi\vert \le R$ or $\operatorname{diam}{\Psi} \leq R$) 
 if  
\begin{enumerate}\label{regBUPU1}
\item  $\psi_0$ is compactly supported in $B_R(0)$. \footnote{$B_R(0)$ is the ball of radius $R > 0$ around zero in $\Rdst$.} 
\item  $\sum_{k \in \Zdst}  \psi_k(x)  
= \sum_{k \in \Zdst}  \psi_0 (x - \gamma k )   \equiv 1$ on $\Rdst$.
\end{enumerate}
\end{definition}
Usually it is assumed that $\psi_0(x) \geq 0$.

\section{Continuous functions that vanish at infinity}
\label{sec:CO}

The uniform or sup-norm of functions on $\Rd$ is defined by $ 
\Vert f \Vert_{\infty} = \sup_{t\in\Rd} \vert f(t) \vert.$ 

Observe that $\CbRd$, the space of all bounded, continuous, complex-valued
functions on $\Rdst$ is a Banach algebra with respect to this
norm and pointwise multiplication. 
It is easy to show that $(\CcRd, \Vert \, \cdot \, \Vert_{\infty})$ is not complete. Its completion in $\CbRdN$, which is the same as the closure within $\CbRdN$, is just the space of continuous functions that vanish at infinity. We denote this space by $\CORdN$. 
For $f \in \CORd$ and $h \in \CbRd$ the pointwise product $f\cdot h$ is again  in $\CORd$.  
In particular,  $\CORdN$ is itself  a (commutative) Banach algebra with respect to pointwise multiplication, with 
\begin{equation}  \label{COCbmult} 
\Vert f \cdot h \Vert_{\infty} \le \Vert f \Vert_{\infty} \, \Vert h \Vert_{\infty}.
\end{equation}

We \emph{define} the space of \emph{bounded measures} $\MbRd$ to be the continuous (Banach space) dual of $\CORdN$. That is, $\MbRd=\COPRd$ consists of all linear and continuous functionals $\mu : \CORd \to \bC$. We write the action of a functional $\mu\in\MbRd$ on a function $f\in\CORd$ as $\mu(f)$. Naturally,
$\MbRd$ is a Banach space with respect to the operator norm,
\begin{equation} \label{MeasNorm1} 
 \Vert \mu \Vert_{\Mbsp} = \sup_{{f\in\CORd,  \, \, \Vert f \Vert_{\infty} \le 1}} \big\vert \mu(f) \big\vert . 
 \end{equation} 

There are two simple and natural examples of bounded measures.
First of all the Dirac measure (or Dirac delta) of the form $\delta_x: f \mapsto f(x)$, $x\in\Rd$.\footnote{What we denote by $\delta_{x}$ is often called the Dirac delta \emph{function} and denoted by $\delta_{x}(t)$ or $\delta(t-x)$ (the argument indicating that it is a ``function'' of, e.g., a time-variable $t$). We do not view the Dirac delta in this way.}
 Their finite linear combinations are called {\it finite
discrete measures} and belong also to $\MbRd$. 

Secondly, any function $g\in\CcRd$ defines a bounded measure $\mu_{g}$ by
\begin{equation}
    \label{eq:1609a} \mu_{g}: \CORd \to\bC, \quad \mu_{g}(f) = \intRd f(t) \, g(t) \, dt, \ \ f\in\CORd.
\end{equation} 
This integral is well defined as $f\cdot g\in\CcRd$.

We mention the following operations that one can do with bounded measures: we define the product of a bounded measure $\mu\in\MbRd$ with a function $h \in \CbRd$ to be the bounded measure given by

\begin{equation}  \label{ptmultmeas} 
\big( \mu \cdot h \big)(f) := \mu ( h \cdot f) \ \ \text{for all} \ \ f \in \CORd.
\end{equation} 
Observe that $\Vert \mu \cdot h \Vert_{\Mbsp} \le \Vert h \Vert_{\infty} \, \Vert \mu \Vert_{\Mbsp},$ and of course associativity. 

Furthermore, we define the complex conjugation of a bounded measure, its flip, translation,  modulation and dilation to be, for any $\mu\in\MbRd$ and $f\in\CORd$,
\begin{align*}
     \overline{\mu} (f) & = \overline{\mu(\overline{f})},\\
     \mu^{\checkmark}\!(f) & = \mu(f^{\checkmark}\!),\\
     \big(T_{x}\mu\big)(f) & = \mu(T_{-x}f), \ x\in \Rd,\\
     \big(E_{\omega}\mu\big)(f) & = \mu(E_{\omega}f), \ \omega\in\Rd, \\
     \big(\alpha_{A} \mu)(f) & = \mu(\alpha_{A^{-1}} f), \ A\in \textnormal{GL}_{\RR}(d). 
\end{align*}

The reader may verify consistency with the corresponding
operators defined on ordinary functions, i.e.\ that for any $g\in\CcRd$
\[ \overline{\mu_{g}} = \mu_{\overline{g}}, \ \ (\mu_{g})^{\checkmark} \!= \mu_{g^{\checkmark}}, \ \ T_{x}\mu_{g} = \mu_{T_{x}g}, \ \ E_{\omega}\mu_{g} = \mu_{E_{\omega}g}, \ \ \alpha_{A}\mu_{g}=\mu_{\alpha_{A}g}.\]


Furthermore, one has the following rather natural rules: 
\[ T_{y}\delta_{x} = \delta_{x+y}, \ \ \delta_{x}^{\checkmark} = \delta_{-x}, \ \ \overline{\delta_{x}} = \delta_{x}, \ \ \delta_{x}\cdot h = h(x) \cdot \delta_{x}.\]

Finally we define $\mu*f$ to be the convolution of a function $f\in \CORd$ with a measure $\mu\in\MbRd$. It is a new function on $\Rd$ given pointwise by 
\begin{equation}  \label{defmuconvf} 
\big(\mu*f\big)(x) = \mu( T_{x}[f^{\checkmark}]) = \big( T_{-x}\mu\big)(f^{\checkmark}\!), \ \ x\in\Rd.\end{equation} 

Observe that $\delta_{x}*f = T_{x}f$. This correspondence is in fact the reason why the ``moving average'' described in (\ref{defmuconvf}) 
makes use of the flip-operator. 

\begin{theorem} \label{th:convolution-CO}
For any $\mu\in\MbRd$ and any $f\in\CORd$ the convolution product $\mu*f$ is a function in $\CORd$. Moreover,
$C_\mu: f \mapsto \mu \ast f$ is a bounded operator
\[ \Vert \mu * f \Vert_{\infty} \le \Vert \mu \Vert_{\Mbsp} \, \Vert f \Vert_{\infty}, \quad 
 f \in \CORd,\]
which  commutes with
translations, i.e.\ $ \mu \ast (T_x f) = T_x (\mu \ast f)$
for all $x \in \Rdst$. Moreover, the operator norm of 
$C_\mu$ equals the functional norm of $\mu$. 
\end{theorem}

One can in fact show that every continuous operator $T:\CORd\to\CORd$ that satisfied the commutation relation $T \circ T_{x} = T_{x}\circ T$ for all $x\in \Rd$ is given by an operator that convolves with some uniquely determined measure $\mu\in\CbRd$. A proof of this statement and Theorem \ref{th:convolution-CO} can be found in the first author's lecture notes.\footnote{See the lectures notes on ``Harmonic and Functional Analysis'' at \newline  \texttt{https://www.univie.ac.at/nuhag-php/home/skripten.php}}
 Such an operator is also called a {\it translation invariant linear system} (TILS). For more on this, see Section \ref{sec:TILS}.

\begin{definition}
Given $f \in \CbRd$  and $\delta > 0 $ we define the
oscillation function
\begin{equation}  \label{oscdef01} 
 \osc_\delta(f) \ofp{x} :=  \max_{|y| \leq \delta} |f(x) - f(x+y)|.  
\end{equation} 
We also define the {\it local maximal function} 
for any $f \in \CbRd$,
\begin{equation}\label{deflocmax01}
  f^{\#}(x) = \max_{|y| \leq 1}  |f(x+y)|, \quad x \in \Rdst.
\end{equation}
\end{definition}

There are a couple of harmless but useful pointwise estimates:
\begin{lemma} \label{elemest03} For any two functions $f,f_{1},f_{2}\in \CbRd$ one has that
\begin{enumerate}
\item[(i)] 
$ \oscd(f) \leq 2 f^\# $; 
\item[(ii)] $ \oscd (f_1 + f_2)  \leq \oscd(f_1) + \oscd(f_2); $
\item[(iii)] $ |f| \leq |g|  \Rightarrow f^\# \leq g^\#;  $ 
\item[(iv)] $ (f_1 + f_2)^\# \leq f_1^\# + f_2^\#; $
\item[(v)] $\oscd(T_x f) = T_x \oscd(f);$
\item[(vi)] $ (T_x f)^\# =   T_x(f^\#). $ 
\end{enumerate}     
\end{lemma}
\begin{proof}
The proof is left as an exercise to the reader.
\end{proof}
Using these relations, the following is a simple observation.

\begin{lemma} \label{useosc02} 
   A function $f \in \CbRd$ is uniformly continuous if and only if 
  \[ \infnorm{\oscd(f)} \to 0 \ \ \text{for} \ \ \delta \to 0.\]
\end{lemma}

For every BUPU $\Psi $ 
we define the spline-type {\it quasi interpolation operator}  
\begin{equation} \label{splinqint1}
f \mapsto \SpPsi f: \quad 
 \SpPsi f (t) = \sum_{n\in\bZ^{d}} f(t_{n}) \psi_{n}(t), \ \ t\in\Rd.
\end{equation} 
\begin{lemma} \label{splineapp2} 
For any regular BUPU $\Psi$ the operator $\SpPsi$ maps $\CORd$ and $\CbRd$ onto itself, respectively,
with $\| \SpPsi f \|_\infty \leq \|f\|_\infty.$   
One has $ \| \SpPsi f    - f \|_\infty \to 0$ as 
$\operatorname{diam}(\Psi) \to 0$
if and only if $f$ is uniformly continuous (e.g.\ $f \in \CORd$).
\end{lemma}
\begin{proof} The first statement follows easily from the fact that all $\psi_{n}$ are continuous and compactly supported together with the assumed properties of the function $f$.
For the second statement note that we only have to do a pointwise estimate between  $f(t)$  and  $\SpPsi f(t) = \sum_{n \in \Zdst} \psi_n(t_{n}) f(t)$,
where $I\subset \bZ^{d}$ is such that $\supp \,\psi_{n} \cap B_\delta(t)\ne \emptyset$ for all $n \in \Zdst$. 
Using the fact that the $(\psi_{n})$ form a partition of unity, we establish that
\begin{align*}  | \SpPsi f(t) -  f(t)|
& \le  \sum_{n \in \Zdst}  |f(t_n) - f(t)| \cdot \psi_n(t) \end{align*}
If $\Psi$ is a BUPU such that $\vert t - t_n \vert \le \delta$ for all $t\in \supp(\psi_n)$, then we find that
\[ | \SpPsi f(t) -  f(t)| \le \oscd(f)(t).\]
As the support of the functions in the BUPU $\Psi$ is made smaller, we write $|\Psi| \to 0$, $\delta$ go to zero. By Lemma \ref{useosc02}
we conclude that $\Vert \SpPsi f  - f \Vert_{\infty} \to 0$ as $\vert K\vert\to 0$.\end{proof}

One important result that we need for later is the following one.
We give a proof of Theorem \ref{th:Mb-by-finite-measure} at the end of 
this section.
\begin{theorem} \label{th:Mb-by-finite-measure} Let $\Psi=(\psi_{n})_{n\in\bZ^{d}}$ be the BUPU as in \eqref{eq:our-bupu}. Every $\mu\in\MbRd$ can be represented by the absolutely norm convergent series $\mu = \sum_{n\in\bZ^{d}} \mu \cdot \psi_{n}$. Moreover,
\begin{equation} \label{sumPsimu} 
\Vert \mu \Vert_{\Mbsp} 
= \sum_{n\in\bZ^{d}} \Vert \mu\cdot \psi_{n} \Vert_{\Mbsp}.
\end{equation} 
\end{theorem}
\begin{corollary}
For any $\mu\in\MbRd$ and any $\varepsilon > 0$ there exists a finite subset $F_0 \subset \bZ^{d}$ such that $\Vert \mu - \sum_{n\in F} \mu\cdot \psi_{n} \Vert_{\Mbsp} < \varepsilon$ for any finite
subset of $\Zdst$ with $F \supseteq F_0$. 
One can think of $ p = \sum_{n \in F} \psi_{n} \in \CcRd $ as a plateau-type
function with     $\Vert \mu -  \mu\cdot p  \Vert_{\Mbsp} < \varepsilon$. 
\end{corollary}
 {\it {Proof of Theorem \ref{th:Mb-by-finite-measure}.}}
For any given $\epso$ let $\varepsilon_{n}>0$, $n\in\bZ^{d}$ be such that $\sum_{n\in\bZ^{d}} \varepsilon_{n} < \varepsilon$. By the definition of $\Vert \mu \cdot \psi_{n} \Vert_{\Mbsp}$ we can find $f_{n}\in \CORd$, $\Vert f_{n} \Vert_{\infty} \le 1$ such that 
\[ \vert \big( \mu \cdot \psi_{n}\big)(f_{n})\vert > \Vert \mu \cdot \psi_{n} \Vert_{\Mbsp} - \varepsilon_{n}.\]
Without loss of generality, we can assume that $\big( \mu \cdot \psi_{n}\big)(f_{n})$ is real-valued and non-negative. For any finite set $F\subset \bZ^{d}$ we define  $f\in\CcRd$ by $f=\sum_{n\in F} f_{n} \cdot \psi_{n}$. We now observe that 
\begin{align*} \mu(f) & = \sum_{n\in F} \mu( f_{n} \cdot \psi_{n}) = \sum_{n\in F} \big( \mu \cdot \psi_{n}\big)(f_{n}) \\
& > \sum_{n\in F} \big( \Vert \mu \cdot \psi_{n} \Vert_{\Mbsp} - \varepsilon_{n} \big) > \Big( \sum_{n\in F} \Vert \mu \cdot \psi_{n} \Vert_{\Mbsp} \Big) - \varepsilon.\end{align*}
By a simple pointwise estimate we find that $\|f\|_\infty \leq 1$.
Thus that for every $\varepsilon>0$ and any finite set $F\subset \bZ^{d}$ there is a function $f\in \Ccsp(\Rd)$, $\Vert f \Vert_{\infty} \le 1$, such that
\[\sum_{n\in F} \Vert \mu \cdot \psi_{n} \Vert_{\Mbsp} \le \mu(f) + \varepsilon\]
This being true for any $\epso$ and any finite set we conclude that 
\[
\sumnZd  \Vert \mu\cdot \psi_{n} \Vert_{\Mbsp} \le \Vert \mu \Vert_{\Mbsp}.\]
Hence  $\sum_{n\in\bZ^{d}} \mu\cdot\psi_{n}$ is absolutely convergent in $\MbRd$. 
Finally we show that $\mu = \sum_{n\in\bZ^{d}} \mu \cdot \psi_{n}$. For any $f\in \CcRd$ we clearly have
\[ \Big(\sum_{n\in \bZ^{d}} \mu \cdot \psi_{n}\Big)(f) = \sum_{n\in F} \big( \mu \cdot \psi_{n})(f) = \mu\big( \sum_{n\in F} \psi_{n} \cdot f\big) = \mu(f),\]
where $F$ is some finite subset of $\bZ^{d}$ that depends on the support of $f$. Since this equality holds for all $\CcRd$ which is dense in $\CORd$, we get $\mu = \sum_{n\in \bZ^{d}} \mu \cdot \psi_{n}$.
The opposite estimate, namely $\Vert \mu \Vert_{\Mbsp} \le \sum_{n\in\bZ^{d}} \Vert \mu\cdot \psi_{n} \Vert_{\Mbsp}$ is clear by the triangle inequality and the completeness of $\MbRdN$.

\section{The Wiener Algebra on $\Rdst$}
 \label{sec:Wiener}

At this point we are in a situation where we can define pointwise multiplication within the Banach algebra $\CORdN$ and we can convolve a measure with a function $\CORd$. Furthermore, we can multiply any measure with a function in $\CbRd$, always together with the corresponding norm estimates. 
 
 But not every function  $f \in \CORd$ defines a measure and it is not
 possible to define the convolution product of two arbitrary functions
 $f_1,f_2 \in \CORd$. Hence it is desirable to reduce the reservoir of 
 ``test functions'' from $\CORdN$ to a smaller one.
 The first step into this direction will be the introduction of ``our new space of test functions'', the Wiener algebra. It is defined as follows: 
\begin{definition}
Given the BUPU $\Psi = (\psi_n)_{n \in \Zdst}$ in \eqref{eq:our-bupu} the {\it Wiener
algebra} $\Wsp(\Rdst)$ 
consist of all continuous functions $f\in \CbRd$ for which the following norm is finite: 
\begin{equation} \label{wienerdef03}
 \|f\|_\Wsp := \sum_{n \in \Zdst}  \|f  \cdot \psi_n\|_\infty < \infty. 
 \end{equation} 
\end{definition}

One can show that the definition does {\it not} depend 
on the particular choice of the BUPU, i.e.\ different BUPUs  $\Psi^1$ or $\Psi^2$ define the same space. Also, $\WRdN$ is a Banach space.
We mention that an equivalent norm on $\WRd$ is given by
\[ \Vert f \Vert_{\Wsp, \, \sqcap} = \sum_{n\in\bZ^{d}} \Vert f \cdot T_{n} \mathds{1}_{[0,1]^{d}} \Vert_{\infty},\]
where $ \mathds{1}_{[0,1]^{d}}$ is the characteristic function on the set $[0,1]^{d}$.
This is the norm still widely used in the literature, and used in H.~Reiter's book
\cite{re68} as an example
of an interesting Segal algebra (and even going back to N.~Wiener's work on Tauberian theorems). Convolution relations for this (and more general Wiener amalgam spaces) are given in \cite{busc84,fe83} and
\cite{fost85}. 

Observe that for any $f\in \WRd$ and $x\in\Rd$ we have, in general, that $\Vert T_x f \Vert_{\Wsp} \ne \Vert f \Vert_{\Wsp}$. We will not need a norm that is strictly isometric with respect to translation. One way to do this is to introduce the \emph{continuous description} of amalgam norms, which has
been given already in \cite{fe77-3}. 

The Wiener algebra relates to the previously considered function spaces as follows:
All functions in $\WRd$ belong to $\CORd$. The space $\CcRd$ is contained
in $\WRdN$ and $\WRd$ is contained in $\CORdN$, both as dense subspaces.
All the inclusions are in fact continuous embeddings. 
Furthermore, just as $\CORd$ and $\CbRd$, the Wiener algebra behaves well with respect to multiplication.

\begin{lemma} \label{le:1109a} \label{le:1109b}
\begin{enumerate}
    \item[(i)] The Wiener algebra $\Wsp(\Rn)$ is continuously embedded into $\CbRd$ and $\CORd$. Specifically, one has that
\[ \Vert f \Vert_{\infty} \le \Vert f \Vert_{\Wsp} \ \ \text{for all} \ \ f\in \Wsp(\Rd).\]
\item[(ii)] The Wiener algebra is an ideal of $\CbRd$ with respect to pointwise multiplication. In fact, for any $h\in \CbRd$ and $f\in \Wsp(\Rd)$ one has that
\[ \Vert h \cdot f \Vert_{\Wsp} \le \Vert h \Vert_{\infty} \, \Vert f \Vert_{\Wsp}.\]
\item[(iii)]The Wiener algebra is a Banach algebra with respect to pointwise multiplication. For any $f,h\in\Wsp(\Rd)$ we have that $\Vert h \cdot f \Vert_{\Wsp} \le \Vert h \Vert_{\Wsp} \, \Vert f \Vert_{\Wsp}$.
\end{enumerate}
\end{lemma}
\begin{proof}
(i). By assumption we have  $1 = \sum_{n\in\bZ^{d}} \psi_{n}(x)$ for all $x\in \Rd$. Hence 
\begin{align*}
    \sup_{x\in\Rd} \vert f(x) \vert = \sup_{x\in\Rd} \vert \sum_{n\in\bZ^{d}} f(x) \, \psi_{n}(x) \vert \le \sum_{n\in\bZ^{n}} \Vert f \cdot \psi_{n} \Vert_{\infty} = \Vert f \Vert_{\Wsp} < \infty,
    \quad \forall   f\in \Wsp(\Rd).
\end{align*}
(ii). Let $h$ and $f$ be as in the statement. It follows from the easy estimate
\[ \sum_{n\in\bZ^{d}} \Vert h \cdot f \cdot \psi_{n}\Vert_{\infty} \le \Vert h \Vert_{\infty} \,\sum_{n\in\bZ^{d}} \Vert f  \cdot \psi_{n}\Vert_{\infty} = \Vert h \Vert_{\infty} \, \Vert f \Vert_{\Wsp}. \]
(iii). This follows by (i) and (ii).
\end{proof}

\begin{lemma} \label{le:translation-wiener}
The translation and the modulation operator are continuous on the Wiener algebra $\WRdN$.
In fact,
\[ \Vert T_{x} f \Vert_{\Wsp} \le 4^{d} \, \Vert f \Vert_{\Wsp} \ \ \text{and} \ \ \Vert E_{\omega} f \Vert_{\Wsp} = \Vert f \Vert_{\Wsp} \ \ \text{for all} \ \ x,\omega\in\Rd, \ f\in\WRd. \]
Moreover, the dilation by an invertible $d\times d$ matrix $A$, $\alpha_{A}f(t) = \vert \det(A) \vert^{1/2} f(At)$ is a continuous operator on $\WRd$ for each such $A$. 
\end{lemma}
\begin{proof}
The relation for the modulation operator is trivial. For the translation operator we have to work a bit harder. First, observe that for any $t,x\in \Rd$ we have
\[ \wedge(t+x) = \wedge(t+x) \cdot 1 = \wedge(t+x) \cdot \sum_{k\in F} \wedge(t-\tfrac{k}{2}),\]
where $F$ is a finite subset of $\bZ^{d}$. In fact, it can be taken to have $4^{d}$ summands. It is helpful to make  a sketch of the situation in the $1$- and $2$-dimensional setting.
With this equality we achieve the desired result as follows,
\begin{align*}
    \Vert T_{x} f \Vert_{\Wsp} & = \sum_{n\in\bZ^{d}} \Vert T_{x} f \cdot \psi_{n} \Vert_{\infty} = \sum_{n\in\bZ^{d}} \sup_{t\in\Rd} \big\vert f(t) \, \wedge(t+x-\tfrac{n}{2}) \big\vert \\
    & = \sum_{n\in\bZ^{d}}\sup_{t\in\Rd} \big\vert f(t) \, \sum_{k\in F} \wedge(t+x-\tfrac{n}{2})\,  \wedge(t-\tfrac{n-k}{2}) \big\vert \\
    & = \sum_{n\in\bZ^{d}} \sup_{t\in\Rd} \big\vert f(t) \, \sum_{k\in F} \wedge(t+x-\tfrac{n+k}{2})\, \wedge(t-\tfrac{n}{2}) \big\vert \\
    & \le \#F \, \Vert \wedge \Vert_{\infty}  \, \sum_{n\in\bZ^{d}} \sup_{t\in\Rd} \big\vert f(t) \,  \, \wedge(t-\tfrac{n}{2}) \big\vert = 4^{d} \, \Vert f \Vert_{\Wsp}.
\end{align*}
The argument for the continuity of the dilation operator is equivalent to the fact that different BUPUs define equivalent norms on the Wiener algebra. We omit the proof.
\end{proof}

The reader may verify the following statement: 
\begin{lemma}\label{le:W-is-solid}
If $f$ is a function in $\WRd$ and $h\in\CbRd$ is such that $\vert h(t) \vert\le \vert f(t) \vert$ for all $t\in\Rd$, then $h\in\WRd$ and $\Vert h \Vert_{\Wsp} \le \Vert f \Vert_{\Wsp}$.
\end{lemma} 

From Lemma \ref{le:W-is-solid} it is easy to prove the following implications:
if $f$ belongs to the Wiener algebra, then so does its absolute value, $\vert f \vert$, its real and imaginary part $\Re(f)$ and $\Im(f)$, and in case $f$ is real valued, also its positive and negative part $f^{+}$ and $f^{-}$, 
\begin{align*}
& \vert f \vert : t\mapsto \vert f(t) \vert, \ \ \Re(f): t\mapsto \Re(f(t)), \ \ \Im(f):t\mapsto\Im(f(t)),\\
& f^{+}: t\mapsto \tfrac{1}{2} \big( \vert f(t) \vert + f(t)\big) \ \ \ \text{and} \ \ \ f^{-} : t \mapsto \tfrac{1}{2} \big(\vert f(t) \vert - f(t)\big), \ \ t\in\Rd.
\end{align*}

Let us turn to the obstacle that we encountered with the function space $\CORd$: not every $f\in \CORd$ can be embedded into $\MbRd$ and we could not define the  convolution between arbitary functions in $\CORd$. The function space $\WRd$ can be completely embedded into $\MbRd$. Essential in this embedding is the key property of a function in the Wiener algebra to be integrable.
The Riemann integral can be extended from $\CcRd$ to a linear and continuous functional on $\Wsp(\Rd)$. That is,
\begin{equation}
    \label{eq:riemann-on-Wsp}
    I : \Wsp(\Rd) \to \bC, \ I(f) = \int_{\Rd} f(t) \, dt, \ f\in \Wsp(\Rd),
\end{equation}
is a well-defined linear functional satisfying $I(f) = I(T_x f)$,
$x \in \Rd$. Actually,
\begin{equation} \label{eq:int-bounded-on-Wsp} \Big\vert \int_{\Rd} f(t) \, dt \Big\vert = \vert I(f) \vert \le I(\vert f \vert) \le \Vert f \Vert_{\Wsp} \ \ \text{for all} \ \ f\in \Wsp(\Rd). \end{equation}
\noindent\textit{Proof of \eqref{eq:int-bounded-on-Wsp}.} Indeed, if we use the specific BUPU in \eqref{eq:our-bupu}, then we find
\begin{align*}
    & \Big\vert \intRd f(t) \, dt \Big\vert = \Big\vert \intRd \sum_{n\in\bZ^{d}} f(t) \, \psi_{n}(t) \, dt \Big\vert \le \sum_{n\in\bZ^{d}} \intRd \big\vert f(t) \, \psi_{n}(t) \vert \, dt \\
    & = \sum_{n\in\bZ^{d}} \int_{n+\big[-\tfrac{1}{2},\tfrac{1}{2}\big]^{d}} \big\vert f(t) \, \psi_{n}(t) \vert \, dt \le \sum_{n\in\bZ^{d}} \Vert f \psi_{n} \Vert_{\infty} = \Vert f \Vert_{\Wsp}.
\end{align*}


For functions in the Wiener algebra we define the $\Lsp^{1}$-norm to be
\[ \Vert f \Vert_{1} : \WRd \to \RR_{0}^{+}, \ \Vert f \Vert_{1} = \intRd \vert f(t) \vert \, dt.\]

The Riemann integral allows to embed the  
Wiener algebra $\WRd$ into 
$\MbRd$: 
\begin{equation} \label{functomeas1} 
\mu_{k}(f) = \intRd f(t) \, k(t) \, dt, \ \ 
f \in \CORd, k \in  \Wsp(\Rd).
\end{equation} 
It is easy to show that 
$\|\mu\|_\Mbsp \leq \|k\|_\Wsp$ for all $k\in\Wsp(\Rd)$ (combine \eqref{eq:int-bounded-on-Wsp} and Lemma~\ref{le:1109a}) and that
the mapping  $k\mapsto \mu_{k}$ from $\WRd$ into $\MbRd$ is injective. 


With this embedding we define the convolution of two functions in the Wiener algebra: if $f,k\in\WRd$, then their convolution product is defined to be
\begin{equation} \label{eq:conv-of-Wiener} \big(k*f\big)(t) = \big(\mu_{k} *f\big)(t) = \intRd f(t-s) \, k(s) \, ds, \ t\in \Rd. \end{equation}

\begin{lemma}
The convolution defined in \eqref{eq:conv-of-Wiener} turns $\WRdN$ into
a commutative Banach algebra with respect to convolution. In fact, 
\begin{equation} 
\| k \ast f\|_\Wsp \leq  4^{d} \, \|k\|_\Wsp \, \|f\|_\Wsp \ \ \text{for all} \ \ k,f\in \WRd.
\end{equation} 
\end{lemma}
\begin{proof} 
That the function $k*f$ is continuous follows from the fact that for any $f\in \Wsp(\Rd)$ the mapping $t\mapsto T_{t}f$ is continuous from $\Rd$ to $\Wsp(\Rd)$. We can easily establish that the Wiener algebra norm is finite: for all $f,k\in\Wsp(\Rd)$
\begin{align*}
    \sum_{n\in\bZ^{d}} \Vert (k*f) \cdot \psi_{n} \Vert_{\infty} 
    & = \sum_{n\in\bZ^{d}} \, \sup_{t\in\Rd} \Big\vert \intRd f(t-s) \, k(s) \, ds \, \psi_{n}(t) \Big\vert \\
    & \le \intRd \vert k(s) \vert \, \sum_{n\in\bZ^{d}} \sup_{t\in\Rd} \big\vert f(t-s) \, \psi_{n}(t) \big\vert \, ds \\
    & = \intRd \vert k(s) \vert \, \Vert T_{s} f \Vert_{\Wsp} \, dt \le 4^{d} \, \Vert k \Vert_{\Wsp} \, \Vert g \Vert_{\Wsp} < \infty.
\end{align*}
\end{proof}

It is an easy application of Fubini's theorem that establishes the well-known inequality for the convolution in relation to the $\Lsp^{1}$-norm,
\begin{equation} \label{eq:L1-conv-algebra} \Vert k*f \Vert_{1} \le \Vert k \Vert_{1} \, \Vert f \Vert_{1} \ \ \text{for all} \ \ k,f\in\WRd.
\end{equation}

\begin{remark}
This observations opens up the possibility 
to {\it define} $\LiRdN$ within $\MbRdN$
as the closure of (the copy of) $\CcRd$ within $\MbRdN$, avoiding 
measure theory and Lebesgue integration completely. Even the Riemann-Lebesgue Theorem can be derived in this way. We do not
pursue this idea further.
\end{remark}

\begin{remark}
As every function in the Wiener algebra is integrable and uniformly bounded, it follows that $\WRd \subset \LiRd$ and $\WRd \subset \CbRd \subset \LinfRd$. This implies that $\WRd$ is a subspace of all the $\LpRd$-spaces for $p\in [1,\infty]$. Moreover, $\Vert f \Vert_{p} \le \Vert f \Vert_{\Wsp}$ for all $f\in \WRd$ and all $p\in[1,\infty]$. Observe that $\LiRd$, just as $\WRd$, is a Banach algebra with respect to convolution. Unlike $\WRd$ however, $\LiRd$ is \emph{not} a Banach algebra with respect to pointwise multiplication.
\end{remark}

\begin{lemma} \label{BUPptwest2}
A function $f \in \CbRd$ belongs to $\WRd$ if and only if $f^{\#}\in \WRd$ 
and 
\begin{equation} \label{maxnormWRd} 
\Vert f \Vert_{\Wsp} \le \Vert f^{\#} \Vert_{\Wsp} \le 8^{d} \, \Vert f \Vert_{\Wsp} \ \ \text{for all} \ \ f\in \WRd.
\end{equation}
\end{lemma}
\begin{proof}
The upper inequality follows by applying the same method as in the proof of Lemma \ref{le:translation-wiener} where we show that the translation operator is bounded on $\WRd$. As $|f(t)| \le f^{\#}(t)$ for all $t\in \Rd$ the lower inequality follows by Lemma \ref{le:W-is-solid}. 
\end{proof}

\begin{lemma} \label{useosc022}
If  $f \in \WRd$, then $\oscd(f) \in \WRd$
and $\lim_{\delta \to 0} \|\oscd(f)\|_\WRd = 0$.
\end{lemma}
\begin{proof}
By Lemma \ref{elemest03} we have the inequality $\oscd(f) \le 2 f^{\#}$. In Lemma \ref{BUPptwest2} we established that $f\in\WRd$ implies that also $f^{\#}\in\WRd$. It follows from Lemma \ref{le:W-is-solid} that $\oscd(f)\in\WRd$. We leave the second statement as an exercise for the reader.

\end{proof}

$\WRd \subset \CORd$ implies that existence of the usual 
convolution, given by  
\begin{equation} \label{MconvWRd} 
 \big(\mu * f\big)(x) = \mu(T_{x}[f^{\checkmark}]), 
 \quad \mu\in\MbRd, f\in \WRd. 
 \end{equation}
Clearly $\mu \ast f \in \CORd$. For the claim
$\MbRd*\WRd \subset \WRd$ 
 we need a lemma: 
\begin{lemma}  \label{le:1209a}
For every compact set $K$ there exists a constant $c_K > 0$ such that for every function $f \in \CcRd$ with 
$\supp(f) \subseteq K+x$, $x\in \Rd$ one has:
\begin{equation} \label{Westsup1}
\|f\|_\Wsp \leq c_K \, \|f\|_\infty.
\end{equation}
\end{lemma}
\begin{proof}
From the definition of a BUPU it follows that for any given compact set $K$ there is a uniform bounded finite number of functions such that for all $x\in \Rd$  $\supp\,\psi_{n}\cap K \ne \emptyset$. 
Therefore, for any $f\in\Wsp(\Rd)$ with $\supp \, f \subset K +x $
\begin{align*}
    \Vert f \Vert_{\Wsp} & = \sum_{n\in\bZ^{d}} \,\, \Vert f \cdot \psi_{n} \Vert_{\infty} = \sum_{n\in\bZ^{d}} \,\,( \sup_{t\in K+x} \vert f(t) \cdot \psi_{n}(t) \vert) \\
    & \le \Big( \sum_{n \in F_x} \Vert \psi_{n} \Vert_{\infty} \Big)  \,  \Vert f \Vert_{\infty} = c_{K} \, \Vert f \Vert_{\infty},
\end{align*}
where $c_{K}$ is this uniform bound in the number of elements in $F_x$.
\end{proof}


\begin{proposition}
We have $\MbRd \ast \WRd \subset \WRd $ and moreover 
\begin{equation} 
\| \mu \ast f\|_\Wsp \leq  c \, \|\mu\|_{\Mbsp} \, \|f\|_{\Wsp} \ \ \text{for all} \ \ \mu\in\MbRd, \ f\in \WRd.
\end{equation} 
\end{proposition}
\begin{proof}
We use the fact that both $\mu\in\MbRd$ and $f\in\WRd$ have an absolutely convergent series representation if one applied a BUPU to each of them, i.e., $\mu=\sum_{n\in\bZ^{d}} \mu \cdot \psi_{n}$ with $\Vert \mu \Vert_{M} = \sum_{n\in\bZ^{d}} \Vert \mu \cdot \psi_{n} \Vert_{M}$ and $f=\sum_{k\in\bZ^{d}} f \cdot \psi_{k}$ with $\Vert f \Vert_{\Wsp} = \sum_{k\in\bZ^{d}} \Vert f \cdot \psi_{k} \Vert_{\infty}$.
Observe that for each $k,n\in\bZ^{d}$ the function
\[ x\mapsto \big(\mu \psi_{n} * f\psi_{k}\big)(x) = \mu\psi_{n}([T_{x}f\psi_{k}]^{\checkmark}) \]
is continuous and compactly supported, hence an element in $\Wsp(\Rd)$. Furthermore, due to the uniform size of the support of the BUPU $(\psi_{n})$ the functions $\mu \psi_{n} * f\psi_{k}$, $k,n\in\bZ^{d}$ all have support within $K+x$, where $K$ is a fixed compact set and $x$ depends on $k$ and $n$. With the BUPU as in \eqref{eq:our-bupu} $K=[0,1]^{d}$. By Lemma \ref{le:1209a}
we have 
\[ \Vert \mu \psi_{n} * f\psi_{k} \Vert_{\Wsp} \le c_{K} \, \Vert \mu \psi_{n} * f\psi_{k} \Vert_{\infty} \le c_{K} \Vert \mu \psi_{n} \Vert_{M} \, \Vert f \psi_{k} \Vert_{\infty}. \]
Combining these inequalities allows us to deduce the desired estimate:
\begin{align*}
    \Vert \mu * f \Vert_{\Wsp} & = \Big\Vert \big( \sum_{n\in\bZ^{d}} \mu \cdot \psi_{n} \big) * \big( \sum_{k\in\bZ^{d}} f \cdot \psi_{k} \big) \Big\Vert_{\Wsp} \\
    & \le \sum_{k,n\in\bZ^{d}} \Vert (\mu \cdot \psi_{n} ) * ( f \cdot \psi_{k} ) \Vert_{\Wsp} \\
    & \le c_{K} \, \sum_{k,n\in\bZ^{d}} \Vert \mu\psi_{n} \Vert_{M} \, \Vert f \psi_{k} \Vert_{\infty} 
     \\     & 
    = c_{K} \, \Vert \mu \Vert_{M} \, \Vert f \Vert_{\Wsp} < \infty.
\end{align*} 
\end{proof}
For later use we note the following result.
\begin{lemma} \label{le:Wiener-restrict}
For any $m\in\bN$ such that $0<m<d$, the operator
    \[ \mathcal{R}_{m} : \WRd \to \Wsp(\RR^{m}) , \ \mathcal{R}_{m}f(x^{(1)},\ldots,x^{(m)}) = f(x^{(1)},\ldots,x^{(m)},0,\ldots, 0), \ x^{(i)}\in\bR,\]
    is continuous. In fact, $\Vert \mathcal{R}_{m} f \Vert_{\Wsp(\RR^{m})} \le \Vert f \Vert_{\Wsp(\RR^{d})}$ for all $f\in \WRd$.
    \end{lemma}
\begin{proof}
The desired inequality is achieved as follows:
\begin{align*}
    \Vert \mathcal{R}_{m}f \Vert_{\Wsp(\RR^{m})} & = \sum_{n\in\bZ^{m}} \Vert \mathcal{R}_{m} f \cdot \psi_{n}^{(m)} \Vert_{\infty} \\
    & = \sum_{n\in\bZ^{m}} \sup_{t\in \RR^{m}} \vert f(t,0)\cdot \psi_{n}^{(m)}(t)  \vert \qquad \quad (0\in \RR^{d-m})\\
    & \le \sum_{n\in\bZ^{d}} \sup_{t\in \RR^{d}} \vert f(t)\cdot \psi_{n}^{(d)}(t)  \vert = \Vert f \Vert_{\Wsp(\RR^{d})}.
\end{align*}
\end{proof}

\section{The Fourier Transform}
\label{sec:Fourier-transform-on-WFW}

As functions in the Wiener algebra are integrable (in the sense of {\it Riemann!}), we can use $\WRd$ as the domain of the Fourier transform. 

\begin{definition}
For   $f\in \WRd$ we define the \emph{Fourier transform},  
\begin{equation} \label{FourWDef1}
\FT f (s) = \hat{f}(s) = \int_{\Rn} f(t) \, e^{-2\pi i s \cdot t} \, dt , \ \ s\in\Rd. 
\end{equation}
\end{definition}

We mention the following classical result.

\begin{lemma} [Riemann-Lebesgue Lemma]
The Fourier transform is a non-expansive and injective linear  operator from $\WRdN$ into $\CORdN$, i.e.\ 
\begin{equation} \label{FourInfEst}
\|\hatf\|_\infty \leq \Vert f \Vert_{1} \leq \|f\|_\Wsp \ \ \text{for all} \ \ f\in\WRd.
\end{equation}
\end{lemma}
A cornerstone of our approach will be the following formula, which has been called   
\emph{fundamental identity for the Fourier transform} by H.~Reiter: 
\begin{theorem} 
\label{Four-fund0}
%
\begin{equation} \label{fund-Fourier0}
\intRd   f(t) \, \hat g(t) \, dt =  \intRd  \hat f (x) \, g(x) \, dx \ \ \text{for all} \ \ f,g \in \WRd. 
\end{equation}
Equally important is the  
convolution theorem for the Fourier transform
\begin{equation} \label{eq:conv-theorem} 
\widehat{f\ast g} = \hat{f} \cdot \hat{g} 
 \ \ \text{for all} \ \ f,g\in\WRd, 
\end{equation}
\end{theorem} 

\noindent\textit{Proof of \eqref{fund-Fourier0} and \eqref{eq:conv-theorem}.} The Fourier transforms $\hat f$ and $\hat g$ are bounded and continuous. By Lemma \ref{le:1109b} both integrands are in $\WRd$ and thus integrable. 
The relation \eqref{fund-Fourier0} follows via Fubini's theorem (which is easy to prove for
Riemann integrals): 
\begin{equation}
\begin{aligned}
\intRd f(t) \hat g(t) \, dt &=
\intRd f(t)\left(\intRd e^{-2 \pi i x \cdot t} g(x) \,  dx  \right ) \, dt\\
& = 
 \intRd  g(x) \left(\intRd e^{-2\pi i x\cdot t} f(t) \, dt \right ) \, dx\\
&= \intRd  \hat f(x) g(x) \, dx.
\end{aligned}
\end{equation}
The convolution theorem \eqref{eq:conv-theorem} is shown in a similar fashion, making use of the exponential law via the identity
$e^{2 \pi i s \cdot t}=e^{2 \pi i s \cdot (t-y)} e^{2 \pi x \cdot y}. $


The Riemann-Lebesgue lemma tells us that the Fourier transform of a function in the Wiener algebra is a function in $\CORd$. As such, they are not necessarily integrable and we have the same issues with it as in Section \ref{sec:CO} (which lead us to the Wiener algebra). Because we can not guarantee that the Fourier transform of a function in the Wiener algebra is integrable, we  can not always apply the inverse Fourier transform (we also have to show that it is actually a transform which inverts the forward Fourier transform on the given domain), 
\[ \IFT f(t) = \intRd f(s) \, e^{2\pi i s \cdot t} \, dt, \ \ t\in\Rd.\]
Therefore we introduce the 
following Fourier invariant function space: 
\begin{equation} \label{WFWdef} 
\WFWRd = \big\{ f\in\Wsp(\Rn) \, : \,  \hat{f}\in \Wsp(\Rn)\big\}. 
\end{equation} 
This space has been studied by B\"urger in \cite{bu81-1} (using the symbol ${\mathscr{B}_{0}}$). 
It is a Banach space with respect to the natural norm 
$\Vert f \Vert_{\WFW} = \Vert f \Vert_{\Wsp} + \Vert \hat{f} \Vert_{\Wsp}$. 
It is non-trivial and in fact 
dense in $\WRdN$) because it contains the Gauss function and all its shifted and modulated versions. 

The Banach space $\WFW$ is well-suited for the formulation of results in Fourier analysis, such as the Fourier inversion theorem:  
\begin{theorem}
\begin{enumerate}
    \item[(i)] For any $f\in \WFW(\Rd)$ the \emph{Fourier inversion formula} holds pointwise,
\begin{equation}  \label{FourInvWWF} 
f(t) = \IFT\hat f(t) = \int_{\Rn} \hat{f}(s) \, e^{2\pi i s\cdot t} \, ds \ \ \text{for all} \ \ t\in \Rd.
\end{equation} 
\item[(ii)] For any pair of functions $f,g\in\WFW(\Rd)$ the \emph{Parseval identity} holds,
\begin{equation} \label{eq:parseval} 
\intRd \hat{f}(t) \, \overline{ \hat{g}(t)} \, dt = \intRd f(t) \, g(t) \, dt.\end{equation}
\item[(iii)] For any $f,g\in\WFW(\Rd)$ we have the formula $\widehat{f\cdot h} = \hat{f} * \hat{g}$.
\item[(iv)] For any $f\in \WFW(\Rd)$ the Poisson formula holds pointwise:
given $m,d \in\bN_{0}$  with $0\le m \le d$ and any 
non-singular $d\times d$ matrix $A$, 
\begin{equation} \label{eq:poisson-gen} \int_{\bR^{m}} \sum_{k\in \bZ^{d-m}} f ( A( x,k) ) \, dx = \frac{1}{\det(A)} \sum_{k\in \bZ^{d-m}} \hat{f} (A^{\dagger} (0,k)),\end{equation}
where $A^{\dagger}$ is the inverse transpose of the matrix $A$.
\end{enumerate}
\end{theorem}
\begin{proof} We only prove (i), 
starting from  the fundamental identity of Fourier analysis, \eqref{fund-Fourier0}.  
Denote by $g_{0}$ the Gaussian, with  $g_{0}(t) = e^{-\pi t\cdot t}$. It has the remarkable property of {\it being invariant under the Fourier transform!}
Consequently, due to properties of the Fourier transform, we have 
\begin{equation}  \label{GaussFour123}
\FT (E_{\omega} \Drho {g_{0}}) =  T_x \Strho g_{0}, \ \  x \in \Rdst, \ \rho > 0. 
\end{equation} 
In (\ref{fund-Fourier0}) we choose  $g=\FT(E_{x}\Drho g_{0})$, and 
 find that for any $f\in \WFWRd$,
\begin{equation} \label{FourInvFast2}
f\ofp{x} = \lim_{\rho\to\infty} \int f\ofp{t} \,
[ \Trans_x \Strho g_0] \ofp{t} \, dt \\
\stackrel{(\ref{fund-Fourier0})}{=} \lim_{\rho\to\infty} \int \fhat\ofp{t} \,
[ E_x \Drho g_{0}] \ofp{t} \, dt
= \int \fhat\ofp{t} \, e^{2\pi itx} \, dt.
\end{equation}
The first limit is justified because 
$ \intRd h(x) \Strho g_0 = h(0)$ for any $h \in \CORd$. If we apply this to $h = T_{-x}f \in \WRd \subset \CORd$, it results in the equality
$$ f(x) = f(0+x) =  T_{-x}f(0) = \lim_{\rho \to 0}
\intRd f(t+x) \Strho g_0(t) dt, $$
which is equal to the expression in the first limit.
For the convergence of the second argument we use the fact that $\hatf \in \WRd$ by the density of $\CcRd$ in $\WRdN$ one can restrict the attention to convergence of $\Drho g_{0}(t) \to 1$ for $\rho \to 0$,
uniformly over compact sets. Details are left to the reader. Reading the left hand side as a function of $x$ it is easily
reinterpreted as $\Strho g_0 \ast f \ofp{x} $, which tends to $f(x)$
uniformly for any $f \in \CORd$, but also in the Wiener
norm for $f \in \WRdN$. A detailed proof of the Fourier invariance of the Gauss function can be found in Example 1.3.3 of \cite{be96} or in
E.~Stein's book (\cite{st70}).


\end{proof}

The Poisson formula \eqref{eq:poisson-gen} is often ``only'' formulated as the Poisson \emph{summation} formula. In this case we set $m=0$ in \eqref{eq:poisson-gen} and obtain
\begin{equation} \label{eq:poisson-sum} \sum_{k\in \bZ^{d}} f ( A k )  = \frac{1}{\det(A)} \sum_{k\in \bZ^{d}} \hat{f} (A^{\dagger} k).\end{equation}
If we apply \eqref{eq:poisson-sum} to the function $E_{\omega}T_{x}f$, $f\in \WFW(\Rd)$, then we find that
\begin{equation} \label{eq:poisson-sum2} \sum_{k\in \bZ^{d}} e^{2\pi i \, A k \cdot \omega} f( A k -x)  = \frac{e^{2\pi i\, \omega \cdot x}}{\det(A)} \sum_{k\in \bZ^{d}} e^{2\pi i \,A^{\dagger} k\cdot x} \hat{f} (A^{\dagger} k-\omega),\end{equation}
for any invertible $d\times d$ matrix $A$, any $x,\omega\in\Rd$ and any $f\in \WFWRd$.
As a concrete example we apply \eqref{eq:poisson-sum2} to the Fourier invariant Gauss function $f(t) = e^{-\pi \, t\cdot t}$, $t\in \Rd$.
This yields the equality
\begin{equation} \label{eq:poisson-gauss} \sum_{k\in \bZ^{d}} e^{-\pi \, ( A k \cdot A k - 2 \, Ak \cdot (x+i\omega))}  = \frac{e^{\pi i\, (x+i\omega)^{2}}}{\det(A)} \sum_{k\in \bZ^{d}} e^{-\pi \, ( A^{\dagger} k \cdot A^{\dagger} k - 2 \, A^{\dagger} k \cdot (\omega+ix) )}. \end{equation}

In principle we could already start a ``simplified distribution theory'' on the basis of the function space
$\WFWRd$, by considering its dual space as the reservoir of generalized functions. 
Since $\WFWRd$ consists only of bounded
and integrable functions that dual space 
already contains Dirac measure (point evaluation functionals) $\delta_{x_0}(f): f(x_0)$,
or integrable as well as bounded or periodic
functions, and even objects like Dirac combs.

However there is one drawback of this space: 
We cannot prove a \emph{kernel theorem}, which is the ``continuous analogue'' of the matrix representation of a linear mapping from $\Rst^n$ to $\Rst^m$ by matrix multiplication with a well defined 
$m \times n$-matrix $\Asp$, see Section \ref{sec:kernel}.
For this we need the \emph{tensor factorization property} of the underlying Banach space of test functions. We will consider this property in the subsequent section by introducing an even smaller space of Banach algebra of test functions, the Segal algebra $\SORdN$\footnote{Also
called {\it Feichtinger's algebra} in the literature.}, which satisfies all the properties that we are interested in.

\section{Tensor factorization}

While the space of functions $\WFW$ is convenient for Fourier analysis, 
it is not suitable enough for our purposes as there is a crucial property we are interested in, namely the tensor factorization property. We explain it here for the space $\WFW$. This notion can be defined analogously for the other spaces we have considered so far, and also for the space $\SO$ that we will define in the next section.

Given two functions, $f^{(1)}, f^{(2)} \in \WFW(\RRm)$ their tensor-product is:
\begin{equation}     \label{tensordef1} 
 \big( f^{(1)}\otimes f^{(2)} \big) (x,y) = f^{(1)} (x) \cdot f^{(2)} (y), \quad (x,y)\in \RRn\times\RRm,
\end{equation} 
This function belongs to $\WFW(\RR^{n+m})$,
and   there is some constant $c>0$ such that 
\begin{equation} \label{tensorEstim1} 
\Vert f^{(1)}\otimes f^{(2)} \Vert_{\WFW(\RR^{n+m})} \le \, c \, \, \Vert f^{(1)} \Vert_{\WFW(\RRn)} \, \Vert f^{(2)} \Vert_{\WFW(\RRm)}, 
\end{equation}
for all $f^{(1)}\in \WFW(\RRn)$ and $f^{(2)}\in \WFW(\RRm)$.

With the help of tensor products we can construct a new Banach space,
the \emph{projective tensor product} of $\WFW(\RRn)$ and $\WFW(\RRm)$,
\begin{align*} \WFW(\RRn)\, \widehat{\otimes}\,\WFW(\RRm) = \Big\{ F \in \WFW(\RR^{n+m}) \, : & \, F = \sum_{j=1}^{\infty} f_{j}^{(1)} \otimes f_{j}^{(2)}, \text{ and where } \\ & \text{ furthermore }
 \sum_{j=1}^{\infty} \Vert f_{j}^{(1)} \Vert_{\Wsp} \, \Vert f_{j}^{(2)} \Vert_{\Wsp} < \infty \Big\}. 
\end{align*}
The norm of a function $F\in \WFW(\RRn)\,\widehat{\otimes}\,\WFW(\RRm)$ is given by
\[ \Vert F \Vert_{\WFW(\RRn)\,\widehat{\otimes}\,\WFW(\RRm)} = \inf \Big\{ \sum_{j=1}^{\infty} \Vert f_{j}^{(1)} \Vert_{\WFW(\RRn)} \, \Vert f_{j}^{(2)} \Vert_{\WFW(\RRm)} \Big\},\]
where the infimum is taken over all possible representations of $F$ of the type $\sum_{j=1}^{\infty} f_{j}^{(1)} \otimes f_{j}^{(2)}$ as described above. 
One can show that 
\begin{equation} \label{eq:1009a}  \WFW(\RRn) \, \, \widehat{\otimes} \, \WFW(\RRm)  \subsetneq \WFW(\RR^{n+m}).   
\end{equation}
That is, the Banach space $\WFW$ does \emph{not} have the tensor factorization property. If so, there would be an equal sign in \eqref{eq:1009a}.

We therefore ask the following: can we find a Banach space of functions that is well-suited for Fourier analysis (such as $\WFW$) and which does have the tensor factorization property.

\section{The Feichtinger algebra}

In this section we answer the question we posed in the last section. We define a Banach space of functions, to be denoted by $\SORd$, that is very well suited for Fourier analysis, it has the tensor factorization property and consequently allows for the formulation of a kernel theorem. It therefore is the Banach space of test functions that we wish for. Figures~\ref{fig:dummy} and \ref{fig:plot} give an overview of this and the other spaces that we have considered so far. In relation to the much used Schwartz space we mention that it is a dense subspace of $\SO$. Functions in $\SO$, however, need not be differentiable.

\begin{figure}[!htb]
\centering
\begin{tikzpicture}[scale=1.4]
\node [rotate=45,anchor=west] (1) at (-0.8,0) {Banach space};
\node [rotate=45,anchor=west] (1) at (-0.3,0) {convolution with};
\node [rotate=45,anchor=west] (1) at (0,0) {bounded measures};
\node [rotate=45,anchor=west] (2) at (0.5,0) {integration};
\node [rotate=45,anchor=west] (3) at (1,0) {embedded into};
\node [rotate=45,anchor=west] (3) at (1.3,0) {its dual space};
\node [rotate=45,anchor=west] (4) at (1.8,0) {domain for the};
\node [rotate=45,anchor=west] (4) at (2.1,0) {Fourier transformation};
\node [rotate=45,anchor=west] (4) at (2.6,0) {Fourier inversion theorem};
\node [rotate=45,anchor=west] (4) at (3.1,0) {Fourier invariant};
\node [rotate=45,anchor=west] (4) at (3.6,0) {Poisson formula};
\node [rotate=45,anchor=west] (5) at (4.1,0) {tensor factorization}; 
\node [rotate=45,anchor=west] (5) at (4.4,0) {property};
\node [rotate=45,anchor=west] (5) at (4.9,0) {kernel theorem};

\node at (-0.8,-0.3) {$\Csp_{0}$};
\node at (-0.8,-0.6) {$\Wsp$};
\node at (-0.8,-0.9) {$\WFW$};
\node at (-0.8,-1.2) {$\SO$};
\node at (-0.15,-0.3) {$\checkmark$};
\node at (-0.15,-0.6) {$\checkmark$};
\node at (-0.15,-0.9) {$\checkmark$};
\node at (-0.15,-1.2) {$\checkmark$};
\node at (0.5,-0.3) {--};
\node at (0.5,-0.6) {$\checkmark$};
\node at (0.5,-0.9) {$\checkmark$};
\node at (0.5,-1.2) {$\checkmark$};
\node at (1.15,-0.3) {--};
\node at (1.15,-0.6) {$\checkmark$};
\node at (1.15,-0.9) {$\checkmark$};
\node at (1.15,-1.2) {$\checkmark$};
\node at (1.95,-0.3) {--};
\node at (1.95,-0.6) {$\checkmark$};
\node at (1.95,-0.9) {$\checkmark$};
\node at (1.95,-1.2) {$\checkmark$};
\node at (2.6,-0.3) {--};
\node at (2.6,-0.6) {--};
\node at (2.6,-0.9) {$\checkmark$};
\node at (2.6,-1.2) {$\checkmark$};
\node at (3.1,-0.3) {--};
\node at (3.1,-0.6) {--};
\node at (3.1,-0.9) {$\checkmark$};
\node at (3.1,-1.2) {$\checkmark$};
\node at (3.6,-0.3) {--};
\node at (3.6,-0.6) {--};
\node at (3.6,-0.9) {$\checkmark$};
\node at (3.6,-1.2) {$\checkmark$};
\node at (4.25,-0.3) {--};
\node at (4.25,-0.6) {--};
\node at (4.25,-0.9) {--};
\node at (4.25,-1.2) {$\checkmark$};
\node at (4.9,-0.3) {--};
\node at (4.9,-0.6) {--};
\node at (4.9,-0.9) {--};
\node at (4.9,-1.2) {$\checkmark$};
\end{tikzpicture}
\caption{An overview of some of the properties of the four Banach spaces of functions that we consider, $\CORdN$, $\WRN$, $(\WFWRn, \Vert \cdot \Vert_{\WFW})$ and $\SORdN$.}
\label{fig:dummy}
\end{figure}

\begin{figure} 
\centering
\includegraphics[scale = .9]{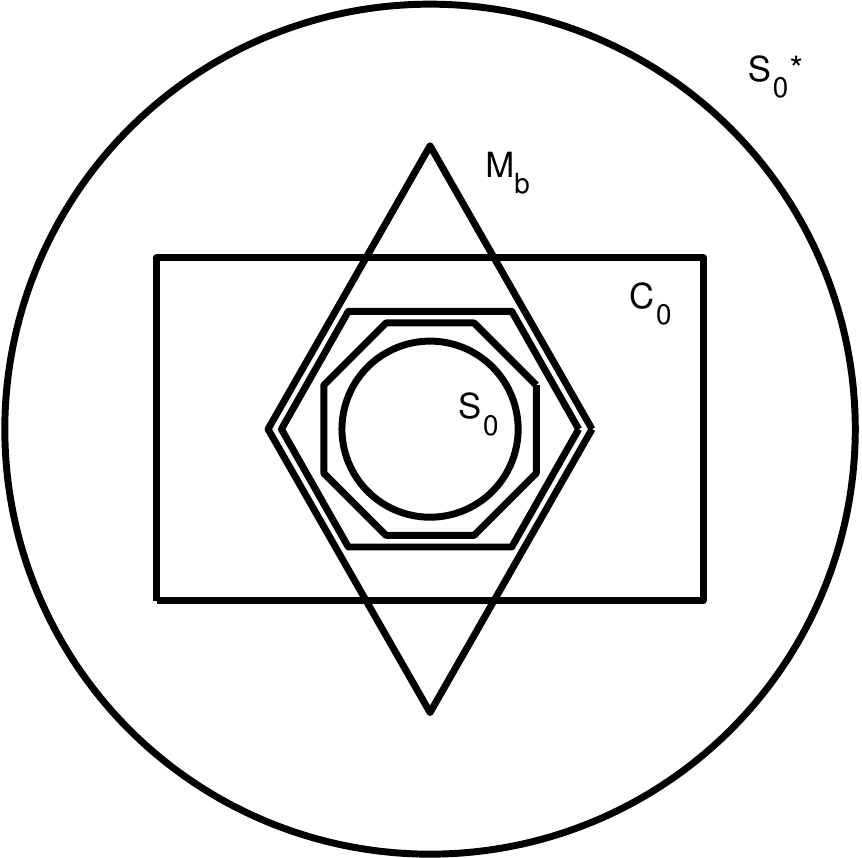}
\caption{The figure visualizes the collection of the spaces that we have considered and their relative (non-)inclusions. In the center we have $\SO$.
Slightly larger, the octagon, is the space $\WFW$. In turn, it is
contained in Wiener's algebra $\Wsp$, which is depicted by a hexagon.
The space $\COsp$ contains all these three spaces, but it is not completely contained in $\Mbsp$. The space $\SOPsp$ contains all spaces. 
The spaces $\SORd$, $\WFWRd$ and $\SOPsp$ are invariant under the Fourier transform. This is represented by a symbol that does not change under rotation by 90 degrees.}
\label{fig:plot}
\end{figure}
 
First we have to introduce the {\it Short-Time Fourier Transform} (or STFT) of a function with respect to a window function $g$.
There are various different assumptions which ensure the pointwise existence (and continuity) of the STFT as a function over the {\it time-frequency plane} or {\it phase space}. We introduct it as follows. 

For a function $g\in \WRd$, the so-called {Gabor window}, which is typically a non-negative, even function concentrated near zero,  we define the {\it short-time Fourier transform} with respect to $g$ of a function $f \in \CbRd$ to be the function
\begin{align*}
& \mathcal{V}_{g} :   \CbRd \to \Cbsp(\Rtdst), \\
& \mathcal{V}_{g}{f}(x,\omega) 
= \int_{\Rd} f(t) \, \overline{g(t-x)} \, e^{-2\pi i \omega \, t} \, dt = \FT( f \cdot \overline{T_{x}g}) (\omega), \ \ x,\omega\in\Rn.\end{align*}

It is easy to see that the definition makes sense for
$g \in \WRd, \, f \in \CbRd$ (still using the Riemann integral). \footnote{It is also a well defined function in $\Cbsp(\Rtdst)$, or for $g,f \in \LtRd$ making use of Lebesgue integration, the usual way of introducing the STFT.} Fix $g_{0}(t) = e^{-\pi \, t\cdot t}$, $t\in \Rd$ to be the Gaussian. 

\begin{definition} The space $\SO(\Rd)$ consists of all functions $f\in \CbRd$ for which $\mathcal{V}_{g_{0}}f$ is a function in  $\Wsp(\Rtd)$.\footnote{In the book \cite{rest00}, and since then, the space $\SO$ has been called the Feichtinger algebra.}
It is endowed with the norm
\[ \Vert \cdot \Vert_{\SO} : \SO(\Rd) \to \RR^{+}_{0}, \ \Vert f \Vert_{\SO} = \int_{\Rtd} \big\vert \big(\mathcal{V}_{g_{0}}f\big)(x,\omega) \big\vert \, d(x,\omega) = \Vert \mathcal{V}_{g_{0}}f \Vert_{1}.\]
\end{definition} 
 Observe that this norm is well-defined, as functions in the Wiener algebra are integrable (see Section \ref{sec:Wiener}).
 
Our goal is to establish the following key result:
\begin{theorem} \label{SOMAIN}
The space $\SORdN$ is a Banach space, which is isometrically invariant under the Fourier transform and time-frequency shifts, and in fact a Banach algebra under convolution as well as multiplication.
\end{theorem} 

We start by observing that 
$\SORd$ is a subspace of Wiener's algebra.
\begin{lemma} \label{le:equiv-norm} \begin{enumerate}
    \item[(i)] The Feichtinger algebra $\SORd$ is a subspace of and continuously embedded into the Wiener algebra $\WRd$.
    \item[(ii)] For any $f\in \SO(\Rd)$ it holds that $\Vert f \Vert_{1} \le \Vert f \Vert_{\SO}$ and $\Vert f \Vert_{\infty} \le \Vert f \Vert_{\SO}$.
\item[(ii)] The mapping $\SORd \to \RR_{0}^{+}, \ f\mapsto \Vert \mathcal{V}_{g_{0}} f \Vert_{\Wsp(\Rtd)}$ is an equivalent norm on $\SORd$. 
\end{enumerate}\end{lemma}
\begin{proof}
Observe that for any $x,s\in\Rd$ we have
\[ \vert f(x) \, g_{0}(s) \vert \le \Vert f \cdot T_{s-x} g_{0} \Vert_{\infty}.\]
Since $f\in\CbRd$, $g_{0}\in\WRd$ and because the translation operator is continuous on $\WRd$, it follows from Lemma \ref{le:1109b} that $f\cdot T_{s-x}g_{0}\in \WRd$ for any $x,s\in\Rd$. 
Furthermore, by assumption $f$ is such that 
\[ (x,\omega)\mapsto\mathcal{V}_{g_{0}}f(x,\omega) = \intRd f(t) \, g_{0}(t-x) \, e^{-2\pi i x\cdot t} \, dt = \mathcal{F}(f\cdot T_{x}g_{0})(\omega) \]
is a function in $\Wsp(\RR^{2d})$. This implies, by Lemma \ref{le:Wiener-restrict}, that for fixed $x\in\Rd$ the function $\omega\mapsto \mathcal{F}(f\cdot T_{x}g_{0})(\omega)$ belongs to $\WRd$ as well. We may therefore apply the Fourier inversion formula, so that, for any $x,s\in\Rd$,
\[ \IFT \FT (f \cdot T_{s-x} g_{0}) = f \cdot T_{s-x} g_{0}.\]
By the Riemann-Lebesgue lemma
\begin{equation} \label{eq:2409c}  \Vert \IFT \FT (f \cdot T_{s-x} g_{0}) \Vert_{\infty} \le  \Vert \FT (f \cdot T_{s-x} g_{0}) \Vert_{\Wsp} =  \sum_{m\in\bZ^{d}} \Vert \FT (f \cdot T_{s-x} g_{0}) \cdot \psi_{m} \Vert_{\infty}.\end{equation}
A combination of the observed facts  yields the inequality
\[ \vert f(x) \, g_{0}(s) \vert \le \sum_{m\in\bZ^{d}} \Vert \FT (f \cdot T_{s-x} g_{0}) \cdot \psi_{m} \Vert_{\infty}.\]
Hence
\[ \sup_{x\in\Rd} \vert f(x) \, g_{0}(s)  \psi_{n}(x) \vert \le  \sum_{m\in\bZ^{d}} \, \sup_{x,\omega\in\Rd}  \vert \FT (f \cdot T_{s-x} g_{0})(\omega) \cdot \psi_{m}(\omega) \psi_{n}(x) \vert.\]
Summing over $n$, and using that the translation operator is continuous on $\WRd$ allows us to deduce that 
\[ \sum_{n\in\bZ^{d}} \Vert f \cdot \psi_{n} \Vert_{\infty} \, \vert g_{0}(s) \vert \le 4^{d} \, \sum_{n,m\in\bZ^{d}} \big\vert \mathcal{V}_{g_{0}} f(x,\omega) \, \psi_{n}(x) \, \psi_{m}(\omega) \big\vert = 4^{d} \Vert \mathcal{V}_{g_{0}} f \Vert_{\Wsp},\]
for any $s\in\Rd$ and $f\in\SORd$. 
It follows that
\begin{equation} \label{eq:2409a} \Vert f \Vert_{\Wsp} \le 4^{d} \, \Vert g_{0} \Vert_{\infty}^{-1} \, \Vert \mathcal{V}_{g_{0}}f  \Vert_{\Wsp} = 4^{d} \, \Vert \mathcal{V}_{g_{0}}f  \Vert_{\Wsp}.\end{equation}
We now show that there exists a constant $c>0$ such that  
\[ \Vert \mathcal{V}_{g_{0}}f  \Vert_{\Wsp} \le c \, \Vert  f  \Vert_{\SO} \ \ \text{for all} \ \ f\in\SO(\Rd).\]
We first establish the following equality: for any $f\in \SORd$ and $x,\omega\in\Rd$
\begin{align} & \int_{\Rtd} \mathcal{V}_{g_{0}}f(t,\xi) \, \overline{\mathcal{V}_{g_{0}} [E_{\omega}T_{x}g_{0}](t,\xi)} \, d(t,\xi) \nonumber \\
& = \int_{\Rtd} \FT(f\cdot T_{t} g_{0})(\xi) \, \overline{\FT([E_{\omega}T_{x}g_{0}]\cdot T_{t} g_{0})(\xi)} \, d(t,\xi) \nonumber \\
& \stackrel{\eqref{eq:parseval}}{=} \int_{\Rtd} (f\cdot T_{t} g_{0})(s) \, \overline{([E_{\omega}T_{x}g_{0}]\cdot T_{t} g_{0})(s)} \, d(t,s) \nonumber \\
& = \intRd f(s) \, \overline{E_{\omega}T_{x}g_{0}(s)} \, \intRd g_{0}(s-t) \, g_{0}(s-t) \, dt \, ds = 2^{-d/2} \, \mathcal{V}_{g_{0}}f(x,\omega). \label{eq:STFT-biorth}
\end{align}
The use of \eqref{eq:parseval}
 is justified as both $f\cdot T_{t}g_{0}$ and $\FT(f\cdot T_{t} g_{0})$ are functions in the Wiener algebra (as already establish earlier in the proof). 
We now observe the following: 
\begin{align*}
& \Vert \mathcal{V}_{g_{0}} f \Vert_{\Wsp} = \sum_{m,n\in\bZ^{d}} \sup_{x,\omega} \big\vert \mathcal{V}_{g_{0}} f(x,\omega) \, \psi_{n}(x) \, \psi_{m}(x) \, \big\vert \\
& \stackrel{\eqref{eq:STFT-biorth}}{=} 2^{d/2} \sum_{m,n\in\bZ^{d}} \sup_{x,\omega} \Big\vert \int_{\Rtd} \mathcal{V}_{g_{0}}f(t,\xi) \, \overline{\mathcal{V}_{g_{0}} [E_{\omega}T_{x}g_{0}](t,\xi)} \, d(t,\xi)
 \, \psi_{n}(x) \, \psi_{m}(\omega) \, \Big\vert \\
 & \le 2^{d/2} \int_{\Rtd} \vert \mathcal{V}_{g_{0}}f(t,\xi) \, \Vert  T_{t,\xi} \mathcal{V}_{g_{0}}g_{0} \Vert_{\Wsp} \, d(t,\xi) \\
 & \le 2^{9d/2} \Vert \mathcal{V}_{g_{0}} g_{0} \Vert_{\Wsp} \int_{\RR^{2d}} \vert \mathcal{V}_{g_{0}}f(t,\xi) \vert \, d(t,\xi) = 2^{9d/2} \, \Vert \mathcal{V}_{g_{0}} g_{0} \Vert_{\Wsp} \, \Vert f \Vert_{\SO}.
\end{align*}
The second equality follows by the boundedness of the translation operator on the Wiener algebra. 
Combining the just established inequality with \eqref{eq:2409a} yields
\[ \Vert f \Vert_{\Wsp} \le 2^{13d/2} \, \Vert \mathcal{V}_{g_{0}} g_{0} \Vert_{\Wsp} \, \Vert f \Vert_{\SO} \ \ \text{for all} \ \ f\in\SORd.\]
Furthermore, we have just established that
\[ \Vert \mathcal{V}_{g_{0}} f \Vert_{\Wsp} \le 2^{9d/2} \, \Vert \mathcal{V}_{g_{0}} g_{0} \Vert_{\Wsp} \, \Vert f \Vert_{\SO}  \ \ \text{for all} \ \ f\in\SORd.\]
The inequality $\Vert f \Vert_{\SO} \le \Vert \mathcal{V}_{g_{0}}f \Vert_{\Wsp}$ is clear from \eqref{eq:int-bounded-on-Wsp}. We have thus shown (i) and (iii). In order to show (ii) we replace \eqref{eq:2409c} with the inequality
\[ \Vert \IFT \FT(f\cdot T_{s-x}g_{0}) \Vert_{\infty} \le \Vert \FT(f\cdot T_{s-x}g_{0}) \Vert_{1}, \]
and make similar steps as before. We then obtain the estimate
\[ \vert f(x) g_{0}(s) \vert \le \int_{\Rd} \big\vert \mathcal{V}_{g_{0}}f(s-x,\omega) \big\vert \, d\omega \ \ \text{for all} \ \ x,s\in\Rd. \]
An integration over $x\in\Rd$ and taking the supremum over $s$ yields 
\[ \Vert f \Vert_{1} \, \Vert g_{0}\Vert_{\infty} \le \int_{\Rtd} \big\vert \mathcal{V}_{g_{0}}f(s-x,\omega) \big\vert \, d(x,\omega) =  \Vert f \Vert_{\SO}.\]
Switching the role of $x$ and $s$ implies the inequality $\Vert f \Vert_{\infty} \le \Vert f \Vert_{\SO}$.
This shows (ii).
\end{proof}

As every function in $\SORd$ belongs to $\WRd$ we can apply the Fourier transform to the space $\SORd$. It turns out that $\SORd$ is invariant under the Fourier transform.

\begin{proposition} \label{pr:SO-F-invariant}
The Fourier transform 
  %
  is an isometric bijection from $\SO(\Rn)$ onto itself, i.e.\  $\Vert \FT f \Vert_{\SO} = \Vert f \Vert_{\SO}$ for all $f\in\SORd$.
\end{proposition} 
\begin{corollary}
$\SO(\Rd)$ is continuously embedded into $\WFWRd$.
\end{corollary}
That $\SORd$ is a proper subspace of $\WFWRd$ was shown by Losert \cite[Theorem 2]{lo80}. Observe that the inclusion $\SORd\subset \WFWRd$ implies that all the statements in relation to the Fourier transform in Section \ref{sec:Fourier-transform-on-WFW} also hold for all functions in $\SORd$. 

\noindent \textit{Proof of Proposition \ref{pr:SO-F-invariant}.}  $\,$ 
First of all $\SORd\subset\WRd$, so that $\FT f$,  is a well-defined function in $\CORd$.
Since $g_{0}\in\WRd$ and $\SORd)\subset \WRd$, we can use the fundamental identity of Fourier analysis to establish the following,
\begin{align*}
     \mathcal{V}_{g_{0}} \hat{f}(x,\omega) & = \intRd \hat{f}(t) \, \overline{g_{0}(t-x)} e^{-2\pi i \omega t} \, dt \stackrel{\eqref{fund-Fourier0}}{=} \intRd f(t) \, \FT\big( \overline{E_{\omega}T_{x}g_{0}}\big)(t)  \, dt \\
     & = e^{-2\pi i x\cdot \omega} \, \mathcal{V}_{g_{0}}{f}(-\omega,x).
\end{align*}
Observe that the phase factor $e^{2\pi i x\cdot \omega}$ and also the change of variable $(x,\omega)\mapsto(-\omega,x)$ are continuous operators on the Wiener algebra, so that also $\mathcal{V}_{g_{0}}\hat{f}$ belongs to $\Wsp(\Rtd)$. Moreover, the operations leave the $\SO$-norm invariant. Indeed, 
\begin{align*}
    & \Vert \hat f \Vert_{\SO} = \int_{\Rtd} \vert  \mathcal{V}_{g_{0}}\hat{f} (x,\omega) \vert \, d(x,\omega) = \int_{\Rtd} \vert e^{-2\pi i x\cdot \omega} \, \mathcal{V}_{g_{0}}{f}(-\omega,x) \vert \, d(x,\omega) \\
    & = \int_{\Rtd} \vert \mathcal{V}_{g_{0}}f (x,\omega) \vert \, d(x,\omega) = \Vert f \Vert_{\SO}.
\end{align*}
The same proof shows that also the inverse Fourier transform maps $\SO(\Rd)$ into itself. It is therefore clear that $\FT$ is a continuous bijection on $\SO(\Rd)$. \\

Concerning the continuity properties of the translation and modulation operator we easily establish the following.

\begin{lemma}
\begin{enumerate}
    \item[(i)] Translation and modulation operators are isometries on $\SO(\Rd)$:  
    \begin{equation}
        \label{eq:TFshift-sO-norm}
    \Vert E_{\omega} T_{x} f \Vert_{\SO} = \Vert f \Vert_{\SO} \ \ \text{for all} \ \ x,\omega\in \Rd \ \ \text{and} \ \ f\in \SO(\Rd).\end{equation}
    \item[(ii)] If $f$ belongs to $\SO(\Rd)$, then so does $\overline{f}$ and $f^{\checkmark}$ and
    \begin{equation} \label{flipisom1} 
\Vert \overline{f} \Vert_{\SO} = \Vert f^{\checkmark} \Vert_{\SO} = \Vert f \Vert_{\SO} \ \ \text{for all} \ \ f\in \SO(\Rd).
    \end{equation} 
\end{enumerate}
\end{lemma}
\begin{proof}
Observe that
\begin{equation} \label{VgTFshift1} 
\mathcal{V}_{g_{0}}(E_{\omega}T_{x}f)(t,s) 
= e^{2\pi i \, x\cdot(\omega-s)} \, \mathcal{V}_{g_{0}}f(t-x,s-\omega). 
\end{equation} 
Since translation and the phase factor  leave the Wiener algebra invariant it follows that $\mathcal{V}_{g_{0}} E_{\omega}T_{x} f\in \Wsp(\Rtd)$. Hence $E_{\omega}T_{x} f\in \SO(\Rd)$ and  moreover  
\begin{align*}
     \Vert E_{\omega}T_{x} f\Vert_{\SO} & = \int_{\Rtd} \vert \mathcal{V}_{g_{0}} (E_{\omega}T_{x} g_{0})(t,s) \vert \, d(x,\omega) \\
    & = \int_{\Rtd} \vert e^{2\pi i \, x\cdot(\omega-s)} \, \mathcal{V}_{g_{0}}f(t-x,s-\omega) \vert \, d(t,s) \\
    & = \int_{\Rtd} \vert \mathcal{V}_{g_{0}}f(t,s) \, d(t,s) = \Vert f \Vert_{\SO} \end{align*} 
for any pair $(x,\omega)\in\Rtd$.
The statement in (ii) is shown in a similar fashion.
\end{proof}  

Just as the Wiener algebra and $\WFW$, also $\SO$ behaves in a nice way with respect to multiplication and convolution.
\begin{lemma}
The Banach space $\SORdN$ is a Banach algebra with respect to pointwise multiplication and convolution. Indeed, for any $f_{1},f_{2}\in \SORd$, the functions $f_{1}\cdot f_{2}$ and $f_{1}*f_{2}$ also belong to $\SORd$ and
\[ \Vert f_{1} \cdot f_{2} \Vert_{\SO} \le  \Vert f_{1} \Vert_{\SO} \, \Vert f_{2} \Vert_{\SO} \ \ \text{and} \ \ \Vert f_{1}* f_{2} \Vert_{\SO} \le \Vert f_{1} \Vert_{\SO} \, \Vert f_{2} \Vert_{\SO}. \]
\end{lemma}
\begin{proof}
Let us first establish $f_{1}\cdot f_{2}$ belongs to $\SORd$. 
\begin{align*}
    \Vert \mathcal{V}_{g_{0}} (f_{1}\cdot f_{2}) \Vert_{\Wsp} & = \sum_{m,n\in\bZ^{d}} \sup_{x,\omega\in \Rd} \big\vert \FT(f_{1}\cdot f_{2}\cdot T_{x}g_{0})(\omega) \, \psi_{n}(x) \, \psi_{m}(\omega)\, \big\vert \\
    & = \sum_{m,n\in\bZ^{d}} \sup_{x,\omega\in \Rd} \big\vert \FT(f_{1}) * \FT( f_{2}\cdot T_{x}g_{0})(\omega) \, \psi_{n}(x) \, \psi_{m}(\omega) \big\vert\\
    & = \sum_{m,n\in\bZ^{d}} \sup_{x,\omega\in \Rd} \Big\vert \intRd \FT( f_{2}\cdot T_{x}g_{0})(\omega-t) \, \hat f_{1}(t) \, dt    \, \psi_{n}(x) \, \psi_{m}(\omega) \Big\vert \\
    & \le \sum_{m,n\in\bZ^{d}} \sup_{x,\omega\in \Rd}  \intRd \vert \FT( f_{2}\cdot T_{x}g_{0})(\omega-t) \, \psi_{m}(\omega)\vert \, \vert \hat f_{1}(t) \vert\, dt    \, \psi_{n}(x) \\
    & \le \sum_{m,n\in\bZ^{d}} \sup_{x,\omega\in \Rd}  \intRd \sup_{\omega\in\Rd} \vert \FT( f_{2}\cdot T_{x}g_{0})(\omega) \, \psi_{m}(\omega+t)\vert \, \vert \hat f_{1}(t) \vert\, dt    \, \psi_{n}(x)  \\
    & \le 4^{d} \sum_{m,n\in\bZ^{d}} \sup_{x\in \Rd}  \intRd \sup_{\omega\in\Rd} \vert \FT( f_{2}\cdot T_{x}g_{0})(\omega) \, \psi_{m}(\omega)\vert \, \vert \hat f_{1}(t) \vert\, dt \, \psi_{n}(x)  \\
    & \le 4^{d} \Vert \hat{f}_{1} \Vert_{\Wsp} \, \Vert f_{2} \Vert_{\SO} \le 16^{d} \, \Vert f_{1} \Vert_{\SO} \, \Vert f_{2} \Vert_{\SO} < \infty.
\end{align*}
In the third inequality we used the same method as in the proof of Lemma \ref{le:translation-wiener} to get rid of the translation by $x$. The inequality for the convolution follows by the just established inequality, the equality $\FT(f_{1}\cdot f_{2}) = \hat{f}_{1} * \hat{f}_{2}$ and the fact that the Fourier transform is a bijection on $\SO$. We have thus established that $\mathcal{V}_{g_{0}} (f_{1}\cdot f_{2}) \in \Wsp(\Rtd)$ and $\mathcal{V}_{g_{0}} (f_{1}* f_{2}) \in \Wsp(\Rtd)$, i.e., the convolution and pointwise product of $f_{1},f_{2}$ belong to $\SO$ again. Concerning the desired estimates, we find that
\begin{align*}
    \Vert f_{1} * f_{2} \Vert_{\SO} & = \int_{\Rd} \int_{\Rd} \big\vert \FT ([f_{1} * f_{2}]\cdot T_{x} g_{0})(\omega) \big\vert \, dx \, d\omega \\
    & = \intRd \intRd \big\vert \big(f_{1} * f_{2} * E_{\omega} g_{0}\big) (x)\big\vert  \, dx \, d\omega \\
    & \le \Vert f_{1} \Vert_{1} \, \intRd \intRd \vert \big(f_{2} * E_{\omega} g_{0}\big)(x)   \, dx \, d\omega \\
    & = \Vert f_{1} \Vert_{1} \, \Vert f_{2} \Vert_{\SO} \le \Vert f_{1} \Vert_{\SO} \, \Vert f_{2} \Vert_{\SO}.
\end{align*}
The first inequality is an application of \eqref{eq:L1-conv-algebra}. The second inequality follows by Lemma \ref{le:equiv-norm}(ii). The inequality for the pointwise product follows by properties of the Fourier transform as mentioned before.\end{proof}

Among other useful properties of $\SORd$ are the following ones. In particular, $\SO$ has the tensor factorization property. 
\begin{theorem}
\begin{enumerate}
    \item[(i)] For any invertible $d\times d$-matrix $A$ the operator 
    \[ \alpha_{A} : \SORd \to \SORd, \ \alpha_{A} f(x) = \vert\det(A)\vert^{1/2} \, f(Ax), \ \ x\in \Rdst,\]
    is a continuous bijection on $\SORd$.
    \item[(ii)] For any $m\in\bN$ such that $0<m<d$, the operator
    \[ \mathcal{R}_{m} : \SORd \to \SO(\RR^{m}) , \ \mathcal{R}_{m}f(x^{(1)},\ldots,x^{(m)}) = f(x^{(1)},\ldots,x^{(m)},0,\ldots, 0), \ x^{(i)}\in\bR,\]
    is a continuous surjection.
    \item[(iii)] The sampling of a function on $\Rn$ at the integer-lattice points $\bZ^{d}$ 
    \[ \mathcal{R}_{\bZ^{d}} : \SORd  \to \ell^{1}(\bZ^{d}), \ \mathcal{R}_{\bZ^{d}}f(k) = f(k), \ \ k\in\bZ^{d} \]
    is a continuous and surjective operator from $\SO(\Rd)$ onto $\ell^{1}(\bZ^{d})$.
    \item[(iv)] For any $m\in\bN$ such that $0<m<d$ the operator
    \begin{align*}
         & \mathcal{P}_{m} : \SO(\Rn)\to \SO(\Rst^m), \\ 
         & \mathcal{P}_{m} f(x) = \int_{\bR^{n-m}} f(x^{(1)},\ldots,x^{(m)},x^{(m+1)},\ldots,x^{(n)}) \, dx^{(m+1)} \ldots dx^{(n)}, \\
         & x = (x^{(1)},\ldots,x^{(m)})\in \Rm.
    \end{align*}
    is a continuous surjection.
    \item[(v)] The periodization of functions on $\Rn$ with respect to the integer lattice $\bZ^{n}$
    \[ \mathcal{P}_{\bZ^{n}} : \SO(\Rn)\to \Asp([0,1]^{n}), \ \mathcal{P}f(x) = \sum_{k\in\bZ^{n}} f(x+k), \ \ x\in [0,1]^{n},\]
    is a continuous and surjective operator from $\SORd$ onto $A([0,1]^{n})$, the space of all $\bZ^{n}$-periodic functions with absolutely-summable Fourier coefficients.
    \item[(vi)] $\SO(\RRn)\,\widehat{\otimes}\,\SO(\RRm)=\SO(\RR^{n+m})$ for any $n,m\in\NN$.
\end{enumerate}
\end{theorem}
\begin{proof}
We are not in the position to give a proof, as this requires more theory and details about $\SO$ than we are willing to give here. The statements all follow from \cite[Theorem 7]{fe81-1}.
\end{proof}

To highlight the role of $\SO(\Rd)$ among all Banach spaces of functions within $\WRd$, we give the following characterization. It is a direct consequence of \cite[Theorem 7.6]{ja19}

\begin{theorem} For each $d\in\NN$ let $(\Bsp(\Rd), \Vert \cdot \Vert_{\Bsp})$ be a non-trivial Banach space such that $\Bsp(\RRd)\subseteq \Wsp(\RRd)$. If for each $d\in \NN$ the Banach space $\Bsp(\RRd)$ has the properties that
\begin{enumerate}
    \item[(i)] there is a constant $c>0$ such that $\Vert f \Vert_{\Wsp(\Rd)} \le c \, \Vert f \Vert_{\Bsp(\Rd)}$ for all $f\in \Bsp(\Rd)$,
    \item[(ii)] for all $(x,\omega)\in \Rtd$ the time-frequency shift operators $E_{\omega}T_{x}$ is bounded on $\Bsp(\Rd)$ with a uniformly bounded operator norm over all $(x,\omega)\in \Rtd$,
    \item[(iii)] for every invertible $d\times d$-matrix $A$ the operator $f\mapsto f \circ A$ is bounded on $\Bsp(\Rd)$,
    \item[(iv)] the Fourier transform is a bounded operator from $\Bsp(\Rd)$ into $\Wsp(\Rd)$,
    \item[(v)] and $\Bsp(\RRn)\,\widehat{\otimes}\, \Bsp(\RRm) = \Bsp(\RR^{n+m})$ for all $n,m\in\NN$,
\end{enumerate}
then $(\Bsp(\Rn),\Vert \cdot \Vert_{\Bsp}) = \SORdN$ for all $d\in\NN$.
\end{theorem}

\section{The shortcut to distribution theory}

\label{sec:distributions}

In the previous sections we described several Banach spaces of continuous functions on $\Rd$ that have useful properties. Figure \ref{fig:dummy} gives a brief overview. Based on this, we recognize $\SO$ as a useful space of test-functions. It has all the properties that we wish for. We will consider its dual space $\SOPRdN$ as a suitably large reservoir of ``everything else'' that is worth to investigate. We call elements in $\SOPsp$ for \emph{distributions}.

The \emph{shortcut} to distribution theory is here the fact that we have established a useful Banach space as our space of test functions. Hence we do not require the more technical details that are typically needed to properly understand the Fr\'echet space formed by the Schwartz functions. Similarly, the dual space, here the Banach space $\SOPRd$ is also much more convenient that the space of tempered distributions (the dual of the Schwartz space). Ergo, with less mathematical effort we can describe and achieve much of the same type of results that the Schwartz space and the temperate distributions are typically used for.

One of the most important concepts of the dual space is that it is possible to extend operators that act on $\SO$ to operators that act on $\SOPsp$. In particular, the properties of $\SO$ allow us to define the Fourier transform of elements in $\SOPsp$ (this is also possible to do with $\WFW$ and $\WFW'$). Before we get to this, we need to introduce $\SOPsp$ properly.

The dual space   $\SOPRd$, consists of bounded, linear  functionals
$ \sigma : \SORd \to \bC$. It is a Banach  space with respect to the usual functional norm
\begin{equation} \label{SOPnorm1}
\Vert \sigma \Vert_{\SOPsp} = \sup_{{f\in \SORd, \, \, \Vert f \Vert_{\SO} = 1}} \vert \sigma(f) \vert.
\end{equation} 
This topology is often too strong. Another weaker, yet at least as natural topology on $\SOPsp$ is the topology it inherits from $\SO$: we say that a sequence $(\sigma_{n})$ in $\SOPRd$ converges in the \weaks topology towards $\sigma_{0}\in \SOPRd$ exactly if
\begin{equation} \label{wstconvdef1} 
 \lim_{n} \big\vert \big(\sigma_{n}-\sigma_{0}\big)(f)\big\vert = 0 \ \ \text{for all} \ \ f\in \SO(\Rd). 
 \end{equation} 

Now every $h\in\CbRd$ (and many more) defines a distribution $\sigma_{h}\in \SOPRd$ via the injective embedding operator
\begin{equation} \label{embeddfct} 
 \iota: \CbRd\to\SOPRd, \ \iota(k) = \sigma_{h} = f \mapsto \intRd f(t) \, h(t) \, dt.
\end{equation} 

Also, any  $\mu\in\Mbsp(\Rn)$ defines a distribution $\sigma_{\mu}\in\SOPRd$ by the rule
\[ \sigma_{\mu}(f) = \mu(f) \ \ \text{for all} \ \ f\in \SO(\Rd).\]
The mapping $\mu\mapsto\sigma_{\mu}$ provides a continuous  embedding $\Mbsp(\Rd)$ into $\SOPRd$.

\begin{definition} \label{def:SOp-extension}
Assume $T$ is a continuous operator from $\SORd$ into $\SORd$. We say that the operator $\widetilde{T}:\SOPRd\to\SOPRd$ is an extension of $T$ if the following holds,
\begin{enumerate}
    \item[(i)] $\widetilde{T}$ is \weaks-\weaks \ continuous,
    \item[(ii)] $\widetilde{T} \circ \iota (k) = \iota \circ T(k)$ (or, equivalently, $\widetilde{T} \sigma_{k} = \sigma_{\,T k}$\,) for all $k\in \SORd$.
\end{enumerate}
\end{definition}

\begin{lemma} \label{extendedOPS} 
The Fourier transform $\FT$, translation operator $T_{x}$, $x\in\Rd$, modulation operator $E_{\omega}$, $\omega\in\Rd$, and the coordinate transform $\alpha_{A}$, $A\in \textnormal{GL}_{d}(\RR)$ are extended from operators on $\SORd$ to operators on $\SOPsp(\Rd)$ in the following way: for any $f\in \SO(\Rd)$ and $\sigma\in\SOPsp(\Rd)$
\begin{align*}
 \widetilde{\FT} & : \SOPRd\to\SOPRd, \ \big(\widetilde{\FT}\sigma\big)(f) =  \sigma(\FT f), \\
\widetilde{T}_{x} & : \SOPRd\to\SOPRd, \ \big(\widetilde{T}_{x}\sigma\big)(f) = \sigma(T_{-x} f), \\
 \widetilde{E}_{\omega} & : \SOPRd\to\SOPRd, \ \big(\widetilde{E}_{\omega}\sigma\big)(f) =  \sigma(E_{\omega} f),\\
\widetilde{\alpha}_{A} & : \SOPRd\to\SOPRd, \ \big(\widetilde{\alpha}_{A}\sigma\big)(f) = \sigma(\alpha_{A^{-1}} f).
\end{align*}
\end{lemma}
\begin{proof}
We only show the result for the Fourier transform. The statements for the other operators are proven in the same fashion.
We have to show that $\widetilde{F}$ satisfies Definition \ref{def:SOp-extension}. In order to show the \weaks-\weaks continuity, let $(\sigma_{n})$ be a sequence in $\SOPRd$ that converges in the \weaks-sense towards $\sigma_{0}$. We have to show that then also $\widetilde{\FT} \sigma_{n} \xrightarrow{\text{w}^{*}} \widetilde{\FT}\sigma_{0}$. This follows easily from the definition of $\widetilde{\FT}$,
\begin{align*}
    & \lim_{n} \big\vert \big( \widetilde{\FT} \sigma_{n} - \widetilde{\FT} \sigma_{0} \big)(f) \big\vert = \lim_{n} \big\vert \big( \widetilde{\FT} (\sigma_{n} - \sigma_{0}) \big)(f) \big\vert \\
    & = \lim_{n} \big\vert \big( \sigma_{n} - \sigma_{0} \big)(\FT f) \big\vert = 0,
\end{align*}
where the last equality follows by assumption. It remains to show that Definition \ref{def:SOp-extension}(ii) is satisfied. We observe that for all $f,k\in\SORd$
\begin{align*} & \big(\widetilde{\FT} \circ \iota (k)\big)(f) = \big(\iota(k)\big)(\FT f) = \intRd \hat{f}(t) \, k(t) \, dt \\
& \big( \iota\circ \FT(k)\big)(f) = \intRd f(t) \, \hat k(t) \, dt.
\end{align*}
It follows from \eqref{fund-Fourier0} that the latter two integrals are the same, so that $\widetilde{\FT} \circ \iota (k) = \iota\circ \FT(k)$, as desired.
\end{proof}

Consider the Dirac delta,
\[ \delta_{x} : \SORd\to \bC, \quad \delta_{x} (f) = f(x), \ x\in\Rd,\]
It is 
easy to show that 
$\widehat{\delta_{x}} = \widetilde{\FT}\delta_{x} $ is the distribution given by
    \[ \widetilde{\FT}\delta_{x} : \SO(\Rd) \to \bC, \quad  \widetilde{\FT}\delta_{x}(f) = \hat f(x) = \intRd f(t) \, e^{-2\pi i x \cdot t} \, dt.\]
    Or, equivalently, $\widetilde{\FT}\delta_{x} = \iota({e_{x}})$, where $e_{x}\in\CbRd$ is given by $e_{x}(t) = e^{-2\pi i x\cdot t}$.
    This can be formulated as to say that ``the Fourier transform of the Dirac delta distribution at $x$, $\delta_{x}$, is the function $e_{x}(t) = e^{-2\pi i t \cdot x}$''. Or, equivalently, ``the Fourier transform of the function $e_{x}(t) = e^{2\pi i t\cdot x}$, $t\in \Rd$, is the Dirac delta distribution at $x$, $\delta_{x}$''.
\begin{remark} 
This is the characteristic property of the Fourier transform: it maps pure frequencies into Dirac measures and vice versa (see \cite{prve08}, (4.36)).
\end{remark}

Consider now the \emph{Dirac comb} or \emph{Shah distribution} for a given invertible $d\times d$ matrix $A$, it is the element of $\SOPRd$ defined by
\[ \Shah_{A} : \SORd \to \bC, \ \Shah_{A}(f) = \sum_{k\in\bZ^{d}} f(Ak).\]
By definition of $\widetilde{\FT}$ and a use of the Poisson summation formula \eqref{eq:poisson-sum}, one gets 
    \[ \widetilde{\FT}(\Shah_{A}) = \vert \det(A) \vert^{-1} \, \Shah_{A^{\dagger}}.\]
 
We define {\it multiplication} and {\it convolution} of a distribution $\sigma\in\SOPRd$ {\it with a test function} $g\in\SORd$ to be the distribution $\sigma \in \SOPRd$ defined as follows:
\begin{definition} \label{convtestmult}  
\begin{align*} 
& \big( \sigma*g\big)(f) =  \sigma(g^{\checkmark}\!*f) \ \ \text{and} \ \  \big( \sigma \cdot g\big)(f) = \sigma(g\cdot f) \ \ 
\quad\ f\in \SORd.
\end{align*}
\end{definition} 
The definition of the convolution is consistent with the definition 
\[ (\sigma * g)(t) = \sigma(T_{t}g^{\checkmark}), \ \ t\in\Rd.\]
Consequently we have $\SORd \ast \SOPRd \subset \CbRd$,
viewed as a subspace of $\SOPRd$. 
Observe that $\Shah_{A} * g$ equals the $A$-period function in $\CbRd$ given by
\[ \big(\Shah_{A} * g\big)(t) = \sum_{k\in\bZ^{d}} g(t+Ak), \ \ t\in \Rd,\]
where the convergence of the series is uniform and
absolute within $\CbRdN$. 
Furthermore, one can show that 
\begin{equation} \label{eq:2009b} \widetilde{\FT}(\sigma*g) = (\widetilde{\FT} \sigma) \cdot ({\FT} g), \ \ \widetilde{\FT}(\sigma \cdot g) = \widetilde{\FT}\sigma * {\FT}g.\end{equation}
We shall use these relations in Section \ref{sec:shannon}, where we take a look at the Shannon sampling theorem.

\noindent\textit{Proof of \eqref{eq:2009b}.} This follows by the definition of the extended Fourier transform and the convolution theorem: for any $\sigma\in\SOPRd$ and $g,f\in \SORd$
\begin{align*}
    & \big(\widetilde{\FT}[\sigma*g]\big)(f) = (\sigma*g)(\FT f) = \sigma(g^{\checkmark}\!*\FT f) \\
    & = \sigma\big([\FT\IFT g^{\checkmark}]\!*\FT f\big) = \sigma\big(\FT [\IFT g^{\checkmark}\!\cdot f]\big) \\
    & = \widetilde{\FT}\sigma(\FT g \cdot f) = \big(\widetilde{\FT}\sigma \cdot \FT g \big) (f).
\end{align*}
The proof of the other equality is done in the same spirit.

\section{The Kernel Theorem} 
\label{sec:kernel}
The reason why $\WFW(\Rd)$ is not quite good enough to be our Banach space of test functions, is that it does not allow for the formulation of a kernel theorem. For this we have to turn to $\SORd$.

The kernel theorem is the continuous analogue of the matrix representation for linear mappings from $\Rst^n$ to $\Rst^m$, showing that they are represented in a unique way through matrix multiplication. Recalling that such a linear mapping $T$ takes the form $T(\xx) = \mathbf{A} \cdot \xx$
for a column vector $\xx \in \Rst^n$ (matrix-vector multiplication), where the columns $(a_k)\kin$ 
are just the images of the unit vectors $(\ee_k)\kin$ in $\Rst^n$ we find that with the usual convention of using indices describing row and column positions of the entries of a matrix 
we have $a_{j,k} = \langle T(\ee_j), \ee_k \rangle_{\Rst^m}$, with $1\leq j \leq n$ and
$ 1 \leq k \leq m$. 

Even by replacing the unit vectors by Dirac measures one cannot hope to get a ``continuous matrix representation'', resp.\ a description of any given operator (say on $\LtRdN$) as an integral operator, because for example multiplication operators cannot have non-zero contributions outside the main diagonal. But we can formulate (in analogy with the Schwartz Kernel Theorem for tempered distributions) a kernel theorem for $\SO$:
\begin{theorem} \label{kernelFei1}
\begin{enumerate}
\item[(i)] The Banach space of operators  $\mathcal{L}(\SORd,\SOPRd)$ can be identified with 
the space $\SOPRtd$. Specifically, to each operator $T$ there corresponds a unique distribution $K \in \SOPRtd$ such that
\begin{equation} \label{kernelchar}
\big(Tf\big)(g) = K (  f \tensor g) \ \ \text{for all} \ \ f,g\in\SORd.
\end{equation} 
\item[(ii)] The Banach space of operators $\mathcal{L}_{w^{*}}(\SOPRd,\SORd)$ that map weak$^{*}$ convergent sequences in $\SOPRd$ into norm convergent sequences in $\SORd$ can be identified with the space $\SORtd$. Specifically, to each operator $T$ there corresponds a unique function $K\in \SORtd$ such that
\begin{equation} \label{SOkernel} 
 \big(T \sigma\big)(x) =  \intRd K(x,y) \, dy \ \ \text{for all} \ \  \sigma \in \SOPRd, \ x \in \Rdst.
 \end{equation} 
Moreover, one has $K(x,y) =  (T\delta_y)(x) = \delta_x(T(\delta_y))$
for all $x,y \in \Rdst$. 
\end{enumerate}
\end{theorem}

Note that the Hilbert space $\LtRtd$ satisfies
 $\SORtd \hookrightarrow \LtRtd \hookrightarrow  \SOPRtd$   and by the classical characterization of Hilbert-Schmidt   operators on $\LtRtd$  this is an intermediate version  of the kernel theorem. Recall that Hilbert-Schmidt 
 operators are compact operators, and form a Hilbert  space with respect to the sesquilinear form  $$  \langle S,T \rangle_{\HS} :=  \trace(S \ast T^*) $$
and the identification is even unitary at this level. For a proof of Theorem 8 we refer to \cite{feja18}.

What we can see from Theorem \ref{kernelFei1}(ii), in the case of ``regularizing operators'', is that they behave very much like matrices, just with continuous entries. This is quite useful for various reasons. It allows to assign (also in the context of $\SOsp$ and $\SOPsp$) to each operator a Kohn-Nirenberg symbol or (via an additional symplectic Fourier transform) a so-called {\it spreading symbol}. These alternative representations are on $\SOPsp(\TFd)$ 
or $\SOsp(\TFd)$ respectively if and only if the
corresponding kernels are in this space. Again those isomorphisms can be seen as extensions resp.\
restrictions of the Hilbert (Schmidt) case, but we will not have space to discuss this at length here (see \cite{cofelu08}). 

But we would like to point at least to the natural composition law for regularizing operators.
Assume that we have two operators $T_1$ and $T_2$
with kernels in $\SORtd$, denoted by $K_1$ and $K_2$. Clearly the composition $T_2 \circ T_1$ of
these operators belongs again to the operator space
$\LSOPwSO$ and therefore has a kernel $K \in \SORtd$. Not very surprising one can show (easily) that one has: 
\begin{equation} \label{composkernel12} 
K(x,z) = \intRd K_2(x,y)  K_1(y,z) dy, \quad
x,z \in \Rdst.
\end{equation} 

When we want to compose two operators with more general kernels, let us assume that now $T_1,T_2$
are just bounded operators on  $\LtRd$, so 
they belong to $\LLt \subset \LSOSOP$, then
they might not have a representation by
kernels in $\SOsp$ in general and the question is how to ``compose'' the kernels. 
For such cases formula \eqref{composkernel12}  
above cannot be applied directly, but it is possible to combine this with regularization operators to ensure that the actual composition is performed on ``nice kernels''. Of course one takes limits after the composition and reaches in this way better and better approximation (in the $\wst$-sense) to the kernel of the composed mapping\footnote{This is comparable with the multiplication of real numbers which is defined as the limit of products of decimal approximations of the involved real numbers, and taking limits afterwards!}. 

When applied to the Fourier transform with the
continuous, bounded and smooth kernel 
$K_1(s,y) =  e^{- 2 \pi i s y} $ and the 
inverse Fourier transform with kernel 
$K_2(s,x) =  e^{2 \pi i x s} $ one can see that
the resulting operator is the identity operator which can be described by the distribution
$\delta_{\Delta}(F) = \intRd  F(x,x) dx$, for 
$F \in \SORtd$, which should be seen as the
continuous analogue of the Kronecker delta-symbol. 
Viewed rowwise (in the continuous sense) the 
entry is just $\delta_x$ at level $x$, or in other words $T(f)(x) = \delta_x(f) = f(x)$, known as the {\it sifting property } of the {\it Dirac delta} (see for example
\cite{prve08}, or \cite{br86-2}). 

Taking the naive approach and computing 
\ref{composkernel12} for the Fourier kernels and
then applying the exponential law results in the (mathematically strange, but often used by engineers) formula
\begin{equation}  \label{invFourone}
   \intinf  e^{-2 \pi i s t} ds =  \delta(t). 
\end{equation}
Such an integral should of course not be 
viewed as an effective integral, but rather
a rule at the level of symbols which is
equivalent to the (independently verifyable
fact) that $\FT \inv \circ \FT = Id$, e.g.
as operators on $\SORd$ (using true integrals)
or in the spirit of Plancherel's Theorem (by taking limits). 

The setting in Theorem \ref{kernelFei1}(i) is general enough to be applied to many of the
operators arising elsewhere, e.g.\ bounded on any of the
space $\LpRdN$ or even from $\LpRdN$ to some other $\LqRdN$,
for $1 \leq p,q \leq \infty$, because one has
$\SORd \subset \LpRd  \subset \SOPRd$ (with continuous embeddings), for
$p,q \in [1,\infty].$ The book of R.~Larsen (\cite{la71}) describes
such operators as convolution operators 
by suitable quasi-measures. These {\it quasi-measures} (introduced by G. Gaudry, \cite{ga66}) are more general
than the elements of $\SOPRd$ and can only be convolved with
compactly supported functions in the Fourier algebra, i.e.\
the elements of the pre-dual. Moreover, unlike elements of $\SOPsp(\Rd)$
it is not possible to define a Fourier transform, resp.\ a 
corresponding transfer function in the natural way. 
Note however that operators with a kernel in $\SOPsp$ do {\it not form an algebra}, 
because the
range of the space may be larger than the domain. On the other hand, for operators mapping a given space into itself (e.g.\ $\LtRdN$, or even $\SORdN$, etc.) composition is  possible and then it should be true (and can be verified) that the convolution of the corresponding kernels ``somehow makes sense'' (using regularizers) or 
equivalently,
the pointwise product of the associated transfer functions will
be also meaningful (e.g. via pointwise a.e. multipication in $\LinfRd$). 

The kernel theorem is the starting point for many 
alternative descriptions of linear operators, more or less by
a ``change of basis''. One can view the space $\SOPRtd$ as a
(huge) space of operators, which contains a number of interesting
operators, such as the collection of all the TF-shifts
$\pi(\lambda) = E_s T_x, \,\, x,s \in \Rdst.$ 
The so-called spreading representation of the operators 
is a kind of ``Fourier-like'' representation of operators,
where these TF-shifts play the role of the Fourier basis 
for the continuous Fourier transform. This representation
will be called the {\it spreading representation} of operators. For more on this see, e.g., \cite{doto07-1} and \cite{feko98}.

\section{Shannon's Sampling Theorem}
\label{sec:shannon}

The  claim of the classical Whittaker-Kotelnikov-Shannon
Theorem  concerns the recovery of any $\LtR$-function whose a Fourier transform
whose support is contained in the symmetric interval $I = [-1/2,1/2]$ around zero (i.e.\ $\supp(\hat f) \subseteq I$)  from regular
samples of the form $(f(\alpha n))_{n \in \Zst}$ as long as $\alpha \leq 1$ (Nyquist rate).  

The reconstruction can be achieved using the $\SINC$-function, with $\SINC(t) = sin(\pi t)/{\pi t}$, the {\it sinus cardinales}
\footnote{The word ``cardinal'' comes into the picture because of the {\it Lagrange type}
interpolation property of the function $\SINC$:  $\SINC(k)=\delta_{k,0}$.},
which can be characterized as the inverse Fourier transform of the
box-function ${\mathds{1}}_I$, the indicator function of $I$.

It is convenient to apply the following notation:
\begin{equation}\label{BOmega}
\Bsp^2_I := \{ f \, : \, f \in \LtR , \,  \supp(\hat f) \subseteq I \}. 
 \end{equation}
The Sampling theorem can be deduced as follows:
By the usual Fourier series, we know that the functions $(e_{k})_{k\in\bZ}=(e^{2\pi i k s})_{k\in\bZ}$ form an complete orthonormal basis in the Hilbert space $\Ltsp([0,1])$,
resp.\ the space of all  functions from
$\LtR$ with  $\supp(\hat{f})\subseteq I$. Therefore
using the standard inner product  $\langle \cdot , \cdot\rangle$ on $\Lsp^{2}(I)$ we obtain: 
\begin{align*} \hat{f}(s) & = \sum_{k\in\bZ} \langle \hat{f},e_{k} \rangle e_{k}(s) = \sum_{k\in\bZ} \langle \hat{f},e_{k} \rangle e^{2\pi i k s} \mathds{1}_{I}(s). \end{align*}
By  applying the inverse Fourier transform  we obtain
\begin{equation} \label{eq:shannon1} f(t) = \sum_{k\in\bZ} \langle \hat{f},e_{k} \rangle \sinc(t+k), \end{equation}
\begin{align*} \mbox{with} \quad \langle \hat{f},e_{k} \rangle & = \int_{I} \hat{f}(s) \, e^{-2\pi i k s} \, ds = \int_{\RR} \hat{f}(s) \, e^{-2\pi i k s} \, ds = f(-k). \end{align*}
Plugging this into \eqref{eq:shannon1} yields the classical version of the Shannon theorem: 
\begin{equation}  \label{ShannLtv}  f(t) = \sum_{k\in\bZ} f(k) \, \sinc(t-k) \ \ \text{for all} \ \ t\in\RR \ \ \text{and} \ \ f\in B_{I}^{2}.
\end{equation} 

Thanks to the fact that the sampling values are in $\ltsp(\Zst)$ the series is pointwise absolutely convergent, even uniformly, but it is also unconditionally convergent in $\LtRN$. Unfortunately the partial sums are {\it not well localized} due to the poor decay of the $\SINC$-function (which is  in $\LtR$, but not in $\LiR$ or $\SO(\RR)$).  

Consequently one prefers to make use of alternative building blocks at the cost of working at a slight oversampling rate.\footnote{Recall that digital audio recordings are meant to capture all the frequencies up to $20$ kHz and  work with $44100$ samples per second, although the abstract Nyquist criterion would only ask for $2 * 20000 = 40000$ samples per second (to express the Nyquist criterion in a practical form). Clearly the use of this theorem in  a real-time situation requires the reconstruction being well localized in time, in order to cause only minimal delay of the reconstruction process.} 
Let us formulate this more practical version of the Shannon sampling for bandlimited functions in the Wiener algebra. 

For any interval $I\subset \RR$ we set 
$ B^1_I := \{ f \in \Wsp(\Rst) \, : \, \supp(\hatf) \subset I \}.$
One can show that 
$  B^{1}_{I} = \{ f \in \SO(\Rst) \, : \,  \supp(\hatf) \subset I \} = \{ f \in \LiR \, : \,  \supp(\hatf) \subset I \}.$ 
The more practical version of Shannon's Sampling Theorem, now with good localization of the building blocks (rather than the $\sinc$-function) reads as follows.

\begin{theorem}
Let $\beta>0$ be such that $I \subset \tfrac{1}{2}(-\beta,\beta)$ and let $g\in\SOR$ be such that $\hat{g}(s)=1$ for all $s\in I$ and $\supp\,\hat{g} \subset \tfrac{1}{2}[-\beta,\beta]$ and let $\alpha = \beta^{-1}$.  
Then we have
\begin{equation}\label{ShannonWR1}
  f(t) = \alpha \, \sum_{k \in \Zst}  f(\alpha k) g(t-\alpha k)
\ \ \text{for all} \ \ t\in\RR, \quad \forall f\in B^{1}_{I}, 
\end{equation}
with absolute convergence in $\SORN$, $\WRN$, and $\CORN$.
\end{theorem}
It is even possible to require that $g$ has decay like the inverse of 
any given polynomial: given $r \in \Nst$ one can find $g$ such that
$ |g(t)| \leq C (1+|t|)^{-r} $ for a suitable constant $C > 0$. The spectrum
of $g$ is contained in a small open interval around $I$.  

\begin{proof}
The assumption about  $\supp(\hatf) \subset I$ implies that the support of all the shifted copies of $\hatf$, are disjoint to $I$ and even
 to the open interval $( -\beta/2, \beta/2)$. Hence for any (ideally
 smooth) function $g$ as in the theorem  satisfies 
 \begin{equation} \label{ShannonId2} 
  (\Shah_{\beta}*\hat{f} ) \cdot \hat{g} = \hat{f} .
 \end{equation} 
By applying the inverse Fourier transform  we find
\begin{equation}  \label{ShannonId3} 
 f = \alpha \cdot (\Shah_{\alpha} \cdot f) * g
 \end{equation} 
 That is, we reach our goal as follows:
 \begin{align*}
     f(t) & = \big(\alpha \cdot  (\Shah_{\alpha} \cdot f) * g \big)(t)  = \alpha \cdot \big(\Shah_{\alpha} \cdot f\big)(T_{t}g^{\checkmark}) \\
     & = \alpha \cdot \big(\Shah_{\alpha} \big)(f\cdot T_{t}g^{\checkmark})  = \alpha \sum_{k\in\bZ} (f\cdot T_{t}g^{\checkmark})(\alpha \, k) \\
     & = \alpha \sum_{k\in\bZ} f(\alpha \, k) \, g(t-\alpha k).
 \end{align*} 
\end{proof} 

\section{Systems and Convolution Operators}
\label{sec:TILS}

The theory of TILS ({\it translation invariant linear systems}) is an important
subject and most electrical engineering students are exposed to this concept early
on in their studies. Unfortunately one must say that -- due to the lack of appropriate
mathematical descriptions -- the way in which the concepts of an impulse response
respectively a transfer function are introduced only in a rather vague
(but ``intuitive") fashion.
Furthermore, students who want to dig deeper and understand these concepts in more
detail are left alone, because engineering books explaining the relevance
of the subject do not provide more details or justifications later on.
On the other hand the mathematical books who talk about convolution 
do this with a completely different motivation but do not connect to those
problems arising in the engineering context.

The article \cite{fe16} takes the first steps towards a reconciliation of
these two approaches\footnote{But still much more has to be done!} by modelling
translation invariant systems of what is called BIBOS systems (which means
bounded input - bounded output), resp.\ as bounded linear operator
from the Banach space $\CORdN$ into itself, commuting with translations.

By choosing as a domain the space  $\CORd$
and {\it not} the larger space $\CbRd$
of all bounded, continuous, complex-valued functions we avoid indeed
the so-called {\it scandal} in system theory as diagnosed by I.~Sandberg
in a series of paper (see e.g. \cite{sa98,sa01-3,sa04-1,sa05}). Furthermore, we are
in fact able to represent every such system as a convolution operator
by some {\it bounded measure}. In order to do so it is not at all required
to discuss technical details of measure theory, but one can just
{\it call}\footnote{this is well justified by the  Riesz representation
theorem.} the bounded (resp.\ continuous) linear functionals on $\CORdN$
bounded measures (as we also did in Section \ref{sec:CO}).

Unfortunately this setting cannot be used to characterize all the
TILS which are bounded on  $\LtRdN$. 
It is true that every convolution operator of the form
$f \mapsto \mu \ast f, \, f \in \LtRd$ with $\mu \in \MbRd$ extends
to all of $\LtRd$ and satisfies the expected estimate:
$\| \mu \ast f\|_2 \leq \|\mu\|_\MbRd \|f\|_2$,
or alternatively can be described on the Fourier transform side
as $\hat{f} \mapsto  \hat{\mu} \cdot \hat{f}$, where $\hat{\mu} \in \CbRd$,
but not every $\Ltsp$-TILS can be represented in this form.

It is not so difficult to find out (using Plancherel's Theorem) that
the most general TILS on $\LtRdN$ is a pointwise multiplier with
an essentially bounded and measurable function, resp.\ with
some $h \in \LinfRd$. So we can write any such operator in the
form $ f \mapsto T(f) = \FT \inv (h \cdot \hat{f})$, with transfer ``function''
$h \in \LinfRd$. But then one would expect that we can write
$ T(f) = \sigma \ast f$, where $\sigma = \FT \inv (h)$, but normally
no inverse Fourier transform for bounded functions (which are not
integrable or at least square integrable) exists. However, this can be made correct by taking the
inverse Fourier transform in the sense of $\SOPRd$ (as defined in Section \ref{sec:distributions}).

One possible example is the convolution by a chirp signal, which
is a bounded, highly oscillating function of the form $ch(t) = e^{i \pi \alpha |t|^2}$.
For simplicity we choose the value $\alpha = 1$.
The general chirp can be obtained from this one by dilations.
This allows us to 
derive from this also the FT of general chirp signals. 

Recall that the chirp  $ch(t) = e^{i \pi  |t|^2}$ belongs
to $\SOPRd$ and therefore has a Fourier transform in this sense. Moreover, it is in fact 
Fourier invariant, and consequently convolution by $ch$ corresponds
to pointwise multiplication of $\hat{f}$ by $ch(t)$, which is a good
operator on $\LtRdN$, because it is continuous and bounded. 

On the other hand one might expect that one can write the convolution
for any $f \in \LtRd$ as an integral, if not as a Riemann integral
so at least as a Lebesgue integral, because this is the most general
integral (at least for our purposes). Specifically, we would like 
to convolve $ch$ with the $\SINC$-function. But due to the fact that
$|ch(t)| = 1, \forall t \in \Rst$ and the fact that $\SINC \notin \LiR$ 
for no argument this convolution integral exists in the literal sense.
It is however (and of course) possible to approximate $f \in \LtR$
by functions $f_n \in \SOR$, to perform the convolutions $ch \ast f_n$
in the classical way, and then take the limit for $n \to \infty$
(with convergence in the $\Ltsp$-sense). 

There are other scenarios, for example (at least mathematicians)
are interested in linear operators from $\LpRdN$ to $\LqRdN$ of a 
similar nature. All of these  cases are covered by the following
theorem:

\begin{theorem} \label{TILSChar11} 
The Banach space $\HLiSOSOP$ of  all bounded linear operators from $\SORdN$ into $\SOPRdN$ which commute with the action of $\WRd$ or  $\LiRd$ by convolution\footnote{In the terminology of Banach modules we are talking about the fact that both $\SORd$ and $\SOPRd$ are 
Banach modules over the Banach convolution algebra $\LiRdN$, and that we are  interested in the Banach module homomorphisms.}, i.e.\ which satisfy 
\begin{equation}\label{LimodHom}
  T(g \ast f) = g \ast T(f), \quad \forall g \in \LiRd, \, f \in \SORd,
\end{equation}
or equivalently the set of all translation invariant bounded operators 
\begin{equation} \label{TILSSOSOP}
 T(T_x f) =  T_x (T\ofp{f}), \quad \forall x \in \Rdst, \, f \in \SORd,
\end{equation}
can be characterized as the set of all convolution operators of the
form $T: f \mapsto \sigma \ast f$ (given pointwise $[\sigma \ast f](x) = 
\sigma( T_x f \checkm)$) where $\sigma\in \SOPRd$. 
In fact, every such operator maps $\SORdN$ into $\CbRdN$, and the corresponding three norms are  equivalent, i.e.\ $\|\sigma\|_\SOPsp$, or the operator norm of $T$ as operator from $\SORd$ into $\CbRdN$ or into $\SOPRdN$, respectively. Moreover, any such operator can be described on the Fourier transform
side as a {\it Fourier multiplier} with the transfer function
$\widehat{\sigma} \in \SOPRd$, via
\begin{equation} \label{transfer2}
\widehat{T(f)} =  \widehat{\sigma} \cdot \hatf, \quad f \in \SORd.
\end{equation}
\end{theorem}

\section{Further References}

These notes are part of a more comprehensive program running
under the title ``Conceptual Harmonic Analysis'' (see \cite{fe16-1}).
It aims at providing a more integrative approach to Fourier Analysis
and its applications, by emphasizing the connections between discrete
and continuous Fourier transform. The contribution provided by
this article is meant to underline that such a more global approach
to Fourier Analysis, which certainly requires the use of generalized
functions (like Dirac measures, Dirac combs, but also almost
periodic function and their Fourier transforms, etc.) does not 
have to start from 
the theory of Schwartz functions and Lebesgue-integration, or even 
from the Schwartz-Bruhat distributions (see \cite{br61,os75})  and (Haar)-measure theory in the case of LCA groups. Instead,  at least for the Euclidean case, a simplified approach can be provided
on the basis of principles from linear functional analysis 
and the Riemann integral for continuous and well decaying functions
on $\Rdst$. Recall that the use of functional analytic methods as such
appears unavoidable due to the fact that relevant signal spaces are rarely finite dimensional. 

The original paper introducing the Banach space $\SO$ for general locally compact abelian groups is \cite{fe81-2}. At that time it was introduced as a particular Segal
algebra in the spirit of H.~Reiter \cite{re68}, 
in fact the smallest member in the family of all {\it strongly 
character invariant} (meaning in modern terminology: {\it isometrically modulation invariant})) Segal algebras. This minimality property gives a large number of properties of these spaces. It is introduced there in the context of general LCA groups. A comprehensive walkthrough of its important properties (also for general LCA groups) is \cite{ja19}. 

It turned out to be the proper domain for the treatment of the metaplectic 
group by H.~Reiter in \cite{re89}  and even for the treatment of {\it generalized stochastic processes} (see \cite{feho14}). 
Also, it is essential for the development
of a general theory of {\it modulation spaces}, which are nowadays a well
established discipline, even with interesting applications in the theory of
partial or pseudo-differential operators (see e.g.\ \cite{fe03-1}, \cite{fe06}).

From the point of view of {\it coorbit theory} as developed in \cite{fegr89} modulation spaces are associated with the STFT, which can be seen as practically equivalent with the matrix coefficients of a pair of vectors  $f,g$ in the Hilbert space $\LtRdN$ under the {\it Schr\"odinger representation} of the {\it reduced  Heisenberg group.} This makes modulation spaces very suitable for the discussion of operators arising in time-frequency analysis and espezicially in connection 
with Gabor Analysis.

It is this area where the usefulness of the spaces
$\SORdN$ and its dual became apparent again and again. Sometimes these two spaces
are viewed together as a {\it Banach Gelfand Triple} denoted
by $\SOGTrRd$. It has been the experiences especially in this area 
where the ideas about ``well chosen function spaces'' became clear. In the spirit of 
\cite{fe15} the current article describes the Wiener
algebra  $\WCOliRd$ and the Segal algebra $\SORdN$ as the most useful Banach spaces of continuous and
integrable functions. It allows to use ordinary Riemann
integrals in a very natural fashion and also covers more
or less all the classical summability kernels. On the way
to a distribution theoretical description of the Fourier
transform (cf. also the elaborations of J.\ Fischer in this 
direciont, \cite{fi15} and \cite{fi17}) the space 
$\WFWRd = \WRd \cap \FT \WRd$ is a first, intermediate step. 


While the concept of modulation spaces was originally
to define Wiener amalgam spaces on the Fourier tranform
side (in the spirit of the Fourier analytic description
of the classical smoothness spaces like $\BspqRdN$, using
dyadic, smooth partitions of unity) also the Wiener algebra
is a representative of the equally important class of
{\it Wiener amalgam spaces}. 
The general theory of Wiener amalgam spaces is described
in \cite{fost85} (Fournier/Stewart) and \cite{busm81} for
the classical case, where the {\it local component} is 
$\LpG$ and the {\it global component} is $\lqsp(\Zdst)$. 
In \cite{fe83}  much more general ingredients were admitted,
which work as long as the local component has a sufficiently
rich pointwise multiplier algebra in order to generate 
BUPUs which are uniformly bounded in that multiplier algebra.
For $\Bsp = \FLpsp$ it is enough to have boundedness in $\FLiRdN$. 
%

The general description of Wiener's algebra (described among
others in \cite{re68} and \cite{rest00}) is the paper \cite{fe77-2}.
The minimality of $\WCOliRd$ and then $\SORd = \WFLiliRd$ is studied
in \cite{fe81} and \cite{fe87-1}. Since the local behaviour of
$\FT\WRd$ equals that of $\FLiRd$ (this is valid for any Segal algbra). 

The pair consisting of  $\SORdN$ and its dual space $\SOPRdN$ can
also serve as a basis for the treatment of generalized stochastic
processes. This approach is described in \cite{feho14}, based
on the PhD thesis \cite{ho87}  of W.~H\"ormann . 

\section{The relation to the Schwartz Theory}

It is of course legitimate to ask about the relationship of the presented approach to the well established Schwartz Theory of
(tempered) distributions (see \cite{sc57}) which is widely
used for PDE or pseudo-differential operators. 


It was first observed by D.~Poguntke that $\ScRd$ is continuously
and densely embedded into $\SORdN$ and consequently 
$\SOPRdN$ is continuously embedded into $\ScPRd$. It is also clear
that the extended Fourier transform for $\ScPRd$, when restricted
to $\SOPRd$ is just the one defined directly in Lemma \ref{extendedOPS} 
without the use of tempered distributions. 
In practice $\SORd$ and $\ScRd$ resp.\ their duals have very similar properties (except for differentiability issues!), including the
existence of a kernel theorem or regularization via smoothing
and pointwise multiplication, using the relations
\begin{equation} \label{regularSO}
 \left (\SOPRd \ast \SORd \right ) \cdot \SORd \subset \SORd 
\end{equation}
which resembles the well-known relationship 
\begin{equation} \label{regularSc}
 \left (\ScPRd \ast \ScRd \right ) \cdot \ScRd \subset \ScRd. 
\end{equation}
But there are still various good reasons to consider the approach presented in this note. First of all, as mentioned several times,it is technically much less challenging, and so the hope is that it has better chances to be adopted by engineers or physicists. In 
particular for courses on signal processing and systems theory
it might be a good way to go.
For people interested in either numerical approximation of
abstract harmonic analysis the function spaces used should offer good tools for a discussion of the connection between the continuous and the finite discrete setting. Such questions usually do not involve  any differentiation.

We also point out that the advantage of a smaller room of 
distributions is the fact, that all the many invariance properties 
allow to show that one is staying within that smaller area.
In \cite{feko98} it was crucial for the derivation of the
Janssen representation of the Gabor frame operator for general
lattices to show that the distributional kernel describing the
spreading function of that operator is supported by the {\it adjoint lattice}, i.e.\ by a discrete set, and that consequently it is a 
sum of Dirac measures (because there is nothing like a practical derivative of the Dirac Delta in $\SOPRd$!). 
We could also argue, that it is enough to know that for any
$ p \in [1,\infty]$ all its elements in $\LpRd$ have a Fourier transform inside of $\SOPRd$ and not only within
some much larger space like $\ScPRd$. Theorem  \ref{TILSChar11} 
is a good example in that direction. Unlike quasi-measures 
(see \cite{la72}) we also find the transfer function inside of the Fourier invariant space $\SOPRd$,  a proper subspace of the space of quasi-distributions.

\subsection*{Acknowledgments}
The work of M.S.J.\ was carried out during the tenure of the ERCIM 'Alain Bensoussan` Fellowship Programme at NTNU. This project was written while  M.S.J.\ was visiting NuHAG at the University of Vienna. He is grateful for their hospitality.  The senior author was finishing this manuscript
while he was holding a guest position at the Mathematical Institute
of Charles University in Prague.


\end{document}